\documentclass[a4paper]{article}

\usepackage[english]{babel}

\usepackage[utf8]{inputenc}
\usepackage[T1]{fontenc}
\usepackage{lmodern}
\usepackage[backend=bibtex,style=alphabetic,firstinits=true,isbn=false,doi=false,url=false]{biblatex}

\usepackage{amsmath, amsthm, amsfonts, amssymb, color, upgreek,bm}
\usepackage{bbm}
\usepackage{tikz}
\usepackage{graphicx}

\usepackage{wasysym }
\usepackage{mathrsfs}

\usepackage{todonotes}
\usepackage{mathtools,bbm}
\usepackage[normalem]{ulem}
\usepackage{cancel}
\usepackage{verbatim}
\usepackage{tikz}
\usetikzlibrary{matrix}
\allowdisplaybreaks[4]
\usepackage[shortlabels]{enumitem}
\setenumerate{label={\normalfont(\alph*)},topsep=4pt,itemsep=0pt} % Enumerate standaart met (a),(b),...

\usepackage{hyperref}
\usepackage{cleveref}
\usepackage{a4}
\usepackage{multirow}
\usepackage[footnotesize,bf,centerlast]{caption}
\usepackage[small,euler-digits,icomma,OT1,T1]{eulervm}

\usepackage{xcolor,colortbl}
\definecolor{Gray}{gray}{0.80}
\definecolor{LightGray}{gray}{0.90}

\newcommand{\bONE}{\mathbbm{1}}

\renewcommand{\epsilon}{\varepsilon}

\numberwithin{equation}{section}

\newtheorem{theorem}{Theorem}[section]
\newtheorem{lemma}[theorem]{Lemma}
\newtheorem{proposition}[theorem]{Proposition}

\theoremstyle{definition}
\newtheorem{definition}[theorem]{Definition}

\newtheorem{remark}[theorem]{Remark}

\newtheorem{assumption}[theorem]{Assumption}

\newtheorem{condition}[theorem]{Condition}

\title{Large deviations for Cox-Ingersoll-Ross processes with state-dependent fast switching}

\author{
	{ Yanyan Hu$^{a)}$,   Richard C. Kraaij$^{a)}$, Fubao Xi$^{b)}$}\\
\footnotesize$^{a)}${Delft Institute of Applied Mathematics, Delft University of Technology,}\\
\footnotesize{Mekelweg 4, 2628 CD Delft, Netherlands.} \\
\footnotesize$^{b)}${School of Mathematics and Statistics, Beijing Institute of Technology, Beijing 100081, P.R.\ China}\\
}
\date{\today}

\begin{document}

\maketitle
\begin{abstract}
We study the large deviations for Cox-Ingersoll-Ross (CIR) processes with small noise and state-dependent fast switching via associated Hamilton-Jacobi equations. As the separation of time scales, when the noise goes to $0$ and the rate of switching goes to $\infty$, we get a limit equation characterized by the averaging principle. Moreover, we prove the large deviation principle (LDP) with an action-integral form rate function to describe the asymptotic behavior of such systems. The new ingredient is establishing the comparison principle in the singular context. The proof is carried out using the nonlinear semigroup method coming from Feng and Kurtz's book \cite{FK2006}.

%The key feature of such systems is very complexity due to discrete and continuous coexistence and we use averaging principle to reduce the effort of computational complexity. Moreover, we prove the large deviation principle (LDP) with an action-integral form rate function to describe the asymptotic behavior of such systems. The proof of LDP modified the method of  Feng and Kurtz's \cite{FK2006}, due to the non-compact domain of CIR processes and the existence of discrete processes, and we give a detailed proof strategy in the paper. 
\noindent 

\noindent 

\end{abstract}

% primary
%	60F10  	Large deviations for stochastic processes
%   60J25  	Continuous-time Markov processes on general state spaces

% secondary
%	470170  	Semigroups of nonlinear operators
%	60J35  	Transition functions, generators and resolvents
%	49L25  	Viscosity solutions for Hamilton-Jacobi equation

%\tableofcontents

%--------Introduction---------------------------
\section{Introduction}\label{se1}

The classical Cox-Ingersoll-Ross (CIR) process was proposed by John C. Cox, Jonathan E. Ingersoll, and Stephen A. Ross in \cite{CIRb1985, CIRa1985}. It is an important tool for modeling the stochastic evolution of interest rates and has widespread applications in the field of finance, especially in the stock market. In general, it is described as
\begin{equation*}
   \mathrm{d} X(t)=\eta(\mu-X(t))\text{d} t+\theta\sqrt{X(t)}\mathrm{d} W(t),
\end{equation*}
where $X(t)$ stands for the instantaneous interest rate at time $t$;
$\eta>0$ is the rate of mean reversion;
$\mu$ represents the mean of the interest rate;
$\theta>0$ is the standard deviation of the interest rate; $(W(t))_{t\ge0}$ is a real value Brownian motion.

In the real world, motivated by the increasing demand for modeling complex systems, in which structural changes, small fluctuations as well as big spikes coexistence are intertwined, we realize that the classical CIR process is lacking the desired complexity.
 Moreover, an instructive example in a stock market is that equity investors can be classified as belonging to two categories, long-term investors and short-term investors. Long-term investors consider a relatively long time horizon and make decisions based on the weekly or monthly performance of the stock, whereas short-term investors, such as day traders, focus on returns in the short term, daily, or even shorter periods. Their time scales are in sharp contrast, and we call it the \textit{two time-scale systems} or \textit{slow-fast systems}. Hence, we add switching to CIR processes that can have mutual impacts, and if we adjust the frequency of the switching may cause a separation of scale.
In this paper, for $E=(0,\infty)$ and $S=\{1,2,\ldots,N\}$, $N< \infty$, we study CIR processes with small noise and fast switching on $E\times S$
\begin{equation}\label{1eqn_CIR}
	\left\{
	\begin{array}{ll}
		\mathrm{d} X^\varepsilon_n(t)=\eta(\mu(\Lambda^\varepsilon_n(t))-X^\varepsilon_n(t))\text{d} t+\frac{1}{\sqrt{n}}\theta\sqrt{X^\varepsilon_n(t)}\mathrm{d} W(t),\\
		~~~~~~~~\\
		(X^{\varepsilon}_n(0),\Lambda^{\varepsilon}_n(0))=(x_0,k_0)\in
		E\times S,
	\end{array}
	\right.
\end{equation}
where the fast process $\Lambda^{\varepsilon}_n(t)$  is a jumping process on $S$ satisfying
\begin{equation}\label{1eqn_fasting_switch}
	\mathbb{P} (\Lambda^{\varepsilon}_n(t+\triangle)=j\mid \Lambda^{\varepsilon}_n(t)=i, X^{\varepsilon}_n(t)= x)=
	\begin{cases}
		\frac{1}{\varepsilon}q_{ij}(x)\triangle+ \circ(\triangle) ,&\mbox{if $j\neq i$,}\\
		1+\frac{1}{\varepsilon}q_{ij}(x)\triangle +\circ(\triangle),&\mbox{if $j=i$,}
	\end{cases}
\end{equation}
for $\Delta >0$, $i,j \in S$, $x\in E$, and $\varepsilon>0$ is a small parameter. 

%\textcolor{red}{Averaging principle?}
The key feature of such slow-fast systems is that the fast process reaches its equilibrium state at much shorter time scales while the slow system effectively remains unchanged. The local equilibration phenomenon allows the approximation of the properties of the slow system by averaging out the coefficients over the local stationary distributions of the fast process. Such approximations yield a significant model simplification and are mathematically justified by establishing an appropriate \textit{averaging principle}.
Hence, when $n\to \infty$ and $\varepsilon\to 0$, the system \eqref{1eqn_CIR} and \eqref{1eqn_fasting_switch} is 
averaged under the law of large number scaling, we can identify that the limit equation is
\begin{equation*}
     \mathrm{d}\bar{X}(t)=\eta \bigg(\sum_{i\in S}\mu(i)\pi^x_i(t)-\bar{X}(t)\bigg)\mathrm{d}t, %~~~{F}(\bar{X}(t))=\eta \bigg(\sum_{i\in S}\mu(i)\pi^x_i(t)-\bar{X}(t)\bigg),
\end{equation*}
where $\pi^x(t)=(\pi^x_i(t))_{i\in S}$ is the stationary distribution of fast processes depending on the position of $X^{\varepsilon}_n(t)=x$.

%\textcolor{red}{LDP?} 
In this setting, we need quantification of how
well, the averaging principle applies to a specific problem. One of the ways to quantify this approximation is \textit{large deviation principle} (LDP) of the Markov processes $(X^\varepsilon_n(t),\Lambda^\varepsilon_n(t))$. In the following, we first conduct an informal analysis.
 
Due to the scale separation phenomenon of slow-fast systems, there are two types of LDP. We first have the Donsker-Varadhan LDP for the occupation measures of fast process $\Lambda^\varepsilon_n(t)$ around $\pi(t)$ when $X^\varepsilon_n(t)$ is close to $x(t)$:

\begin{equation}\label{eqn_star1}
  \mathbb{P}\bigg(\frac{1}{\mathrm{d}t}\int_{t}^{t+\mathrm{d}t} \delta_{\Lambda^{\varepsilon}_n(s)} \mathrm{d} s|_{t\geq 0} \approx \pi(t)|_{t\geq 0} 
 \bigg| X^{\varepsilon}_n(t)|_{t\geq 0}=x(t)|_{t\geq 0}\bigg) \backsim \mathrm{e}^{-\frac{1}{\varepsilon}\tilde{I}_x(\pi)},
\end{equation}
where
\begin{equation*} 
{\tilde{I}_x(\pi)=-\inf_{g \gg 0} \int \frac{R_x g}{g} \mathrm{d} \pi},
\end{equation*}
where ${R}_x$ is the generator of a state-dependent switching:
\begin{equation*}
	R_xg(z)=\sum_{j\in S }q_{zj}(x)\left(g(j)-g(z)\right),
\end{equation*}	
where $x(t):=x$.
\par
Furthermore,  we find that the Freidlin-Wentzel LDP of the slow process $X^{\varepsilon}_n(t)$ is obtained under the condition that the fast process reaches $\pi(t)$, and have 
\begin{multline}\label{eqn_star2}
\mathbb{P}\bigg(\dot{X}^{\varepsilon}_n(t)|_{t\geq 0}\approx \dot{\rho}(t)|_{t\geq 0} \bigg | \frac{1}{\mathrm{d}t}\int_{t}^{t+\mathrm{d}t} \delta_{\Lambda^{\varepsilon}_n(s)} \mathrm{d} s|_{t\geq 0}\approx \pi(t)|_{t\geq 0},X^{\varepsilon}_n(s)|_{s\in(0,t]}=x(s)|_{s\in(0,t]}\bigg)\\
    \backsim e^{-n\hat{I}(\rho |\pi)},
\end{multline}
where
\begin{equation*}
\hat{I}(\rho|\pi)= \min_{\dot{\rho}(t) =\sum_{i=1}^{N} v_i\pi_i(t) }\sum^{N}_{i=1}\frac{|v_i-\eta(\mu(i)-x(t))|^2}{2\theta^2x(t)}\pi_i(t).
\end{equation*}

To analyze the system \eqref{1eqn_CIR} and \eqref{1eqn_fasting_switch} from the point of view of a long-term investor, we need to consider the LDP of both fast and slow processes $(X^{\varepsilon}_n(t),\Lambda^{\varepsilon}_n(t))$ at time $t$ to maximize profits. 
That is formally
\begin{equation}\label{eqn_FWDVLDP}
 \begin{split}
\mathbb{P}&\left(\dot{X}^{\varepsilon}_n(t)|_{t\geq 0}\approx \dot{\rho}(t)|_{t\geq 0}, \frac{1}{\mathrm{d}t}\int_{t}^{t+\mathrm{d}t} \delta_{\Lambda^{\varepsilon}_n(s) } \mathrm{d} s|_{t\geq 0}\approx \pi(t)|_{t\geq 0} \bigg| X^{\varepsilon}_n(s)|_{s\in (0,t]}=x(s)|_{s\in (0,t]} \right )\\
&=\mathbb{P}\left(\dot{X}^{\varepsilon}_n(t)|_{t\geq 0}\approx \dot{\rho}(t)|_{t\geq 0} \bigg|  \frac{1}{\mathrm{d}t}\int_{t}^{t+\mathrm{d}t} \delta_{\Lambda^{\varepsilon}_n(s)} \mathrm{d} s|_{t\geq 0}\approx \pi(t)|_{t\geq 0}, X^{\varepsilon}_n(s)|_{s\in[0,t]} =x(s)|_{s\in[0,t]}\right)\\
     &~~~\times \mathbb{P}\left(\frac{1}{\mathrm{d}t}\int_{t}^{t+\mathrm{d}t} \delta_{\Lambda^{\varepsilon}_n(s)} \mathrm{d} s|_{t\geq 0} \approx \pi(t)|_{t\geq 0} \bigg| X^{\varepsilon}_n(s)|_{s\in(0,t]} =x(s)|_{s\in(0,t]}\right) \\
     &= \exp\bigg\{-\left(n\hat{I}(\rho|\pi)+\frac{1}{\varepsilon}\tilde{I}_x(\pi)\right) \bigg \},
\end{split}   
\end{equation}
where in the last line, we use the results of Donsker-Varadhan LDP \eqref{eqn_star1} and Freidlin-Wentzel LDP \eqref{eqn_star2}.
Moreover, from \eqref{eqn_FWDVLDP} and the contraction principle \cite[Theorem 4.2.1]{DZ1998}, we obtain 
\begin{align*}
\mathbb{P}&\left(\dot{X}^{\varepsilon}_n(t)|_{t\geq 0}\approx \dot{\rho}(t)|_{t\geq 0} \bigg| X^{\varepsilon}_n(s)|_{s\in (0,t]}=x(s)|_{s\in (0,t]} \right )\backsim \exp\bigg\{-\inf_{\pi}\left(n\hat{I}(\rho|\pi)+\frac{1}{\varepsilon}\tilde{I}_x(\pi)\right) \bigg \}.
\end{align*} 
For ease of analysis in subsequent steps, we define a set
\begin{equation*}
    G=\left \{\dot{X}^{\varepsilon}_n(t)|_{t\geq 0}\approx \dot{\rho}(t)|_{t\geq 0} \bigg|   X^{\varepsilon}_n(s)|_{s\in(0,t]} =x(s)|_{s\in(0,t]} \right \}.
\end{equation*}
Hence, we get
\begin{equation}\label{eqn_logPF}
    -\log \mathbb{P}(G)\backsim\inf_{\pi}\left(n\hat{I}(\rho|\pi)+\frac{1}{\varepsilon}\  \tilde{I}_x(\pi)\right).
\end{equation}

In \eqref{1eqn_CIR} the intensity of the multiplicative noise and in \eqref{1eqn_fasting_switch} the frequency of the fast random switching may have different ratios $n^{-1}/\varepsilon$
when $n\to \infty$ or $\varepsilon\to 0$. To determine the optimal LDP's convergence speed of $(X^{\varepsilon}_n(t),\Lambda^{\varepsilon}_n(t))$ at time $t$, it is necessary to analyse \eqref{eqn_logPF} in three different scenarios:
\begin{description}
    \item [Case 1 $\varepsilon \ll \frac{1}{n}$:] we have a Donsker-Varadhan type LDP
    \begin{equation*}
  -\varepsilon \log \mathbb{P}(G)\backsim \inf_{\pi}\left(n\varepsilon\hat{I}(\rho|\pi)+\tilde{I}_x(\pi)\right)\to \inf_{\pi}\tilde{I}_x(\pi);
    \end{equation*}
 \item [Case 2 $\frac{1}{n} \ll \varepsilon $:]  we have a Freidlin-Wentzell type LDP
 \begin{equation*}
  -\frac{1}{n}\log \mathbb{P}(G)\backsim \inf_{\pi}\left(\hat{I}(\rho|\pi)+\frac{1}{n\varepsilon}\tilde{I}_x(\pi)\right)\to \inf_{\pi}\hat{I}(\rho|\pi);
    \end{equation*}
    \item [Case 3 $\varepsilon=\frac{1}{n}$:]
    we have the combination of  Donsker-Varadhan 
    LDP and Freidlin-Wentzell LDP
     \begin{align*}
   -\varepsilon \log &\mathbb{P}(G)\backsim \inf_{\pi}\left(\hat{I}(\rho\mid \pi)+\tilde{I}_x(\pi)\right)\\
&=\inf_\pi\bigg(\min_{\dot{\rho}(t) =\sum_{i=1}^{N} v_i\pi_i(t) }\sum^{N}_{i=1}\frac{|v_i-\eta(\mu(i)-x(t))|^2}{2\theta^2x(t)}\pi_i(t)+\Tilde{I}_x(\pi)\bigg).     
\end{align*}
\end{description}

In this paper, we treat the most complex case, namely Case 3, $\varepsilon=\frac{1}{n}$, which the two LDP's  are completed at the same scale. 

For the proof, we use Feng and Kurtz's method based on Hamilton-Jacobi theory and control theory, which has been developed to study LDP
associated with a sequence of Markov processes. Firstly, the advantage of this method is that the operator convergence
treats both the classical Freidlin-Wentzell theory and the Donsker-Varadhan theory within one framework. Secondly, the Feng and Kurtz's method deals with
the difficulties caused by nonlinear operators by viscosity solutions. Inspired by this approach, Peletier and Schlottke \cite{PS2021} studied a stochastic differential equation with finite state fast switching on the flat torus, but the diffusion coefficient was additive noise. Subsequently, Kraaij and Schlottke \cite{KS2020} investigated the LDP of the slow-fast system by giving the abstract generator with uniformly elliptic conditions. However, as mentioned earlier, all the work mainly considered LDP in a regular setting, we now consider a singular setting. Because of this, Euclidean techniques to study the large deviation principle fail. Alternatively, the authors in \cite{DFL2011} take a Riemann point of view to analyze the associated Hamilton-Jacobi equations, and we extend their insights to the two time-scale contexts by adding switching.

To conclude, we investigate the LDP for CIR processes with state-dependent fast-switching, by associated Hamilton-Jacobi and Hamilton-Jacobi-Bellman equations, which are the primary tools we need. 
Our specific technical road map is as follows: we begin with using Skorokhod's representation to give an integral form of the fast-switching process and obtain strong nonnegative solutions of the CIR processes with fast switching by pathwise splicing. Then, we modify the technique introduced in the book \cite{FK2006} of Feng and Kurtz to 
\begin{enumerate}
    \item verify convergence of the sequence of nonlinear operators $H_n$ to a multivalued limit operator $H$. We reduce $H$ to $\mathbf{H}$ by solving an eigenvalue problem, in which we effwctively find an optimal stationary measure most notably; %reducing single valued $\mathbf{H}$ by eigenvalue problem based on  choice an optimal stationary measure;
	\item verify exponential tightness on the ``path-space'' as the CIR process is a diffusion process equation on $(0,\infty)$ with a singularity at $0$ leading to a non-compact space. 
	\item verify the comparison principle for the nonlinear multivalued limiting operator $H$, which is hard to prove but plays a prominent role. We achieve it by connecting viscosity solutions for $H$ to those for $\mathbf{H}$ and prove comparison principle for %replacing $H$ by $\mathbf{H}$, and we first get the comparison of 
 $f-\lambda \mathbf{H}f=h$, $\lambda>0$.
	\item construct a variational representation for $\mathbf{H}$, which gives the rate function with an action-integral from.
\end{enumerate}

%To enhance the paper's readability, we use Figure \ref{1} to explain further this paper's research outlines.

% \begin{figure}[htp]
% 	\centering
% 	\begin{tikzpicture}[>=stealth,xscale=0.8,yscale=0.7]
% 			\fill[blue, draw=black,rounded corners,fill opacity=0.3] (-10,4) rectangle (-7,2.5);
% 			\fill[blue, draw=black,rounded corners,fill opacity=0.3] (-5,4) rectangle (-2,2.5);
% 				\fill[blue, draw=black,rounded corners,fill opacity=0.3] (0,4) rectangle (3,2.5);
% 			\draw[->,very thick] (-8,2.5)--(-5,0.9);		
% 				\draw[->,very thick] (-3.5,2.5)--(-3.5,1);		
% 					\draw[->,very thick] (1,2.5)--(-2,0.9);		
% 		\node at (-8.5,3.5){Operator};
% 	\node at (-8.5,3){convergence};
	
% 	\node at (-3.5,3.5){Exponential};
% 		\node at (-3.5,3){tightness};
		
% 		\node at (1.5,3.5){Comparison};
% 			\node at (1.5,3){principle};
% 		\fill[red, draw=black,rounded corners,fill opacity=0.3] (-5,0) rectangle (-2,1);
% 			\node at (-3.5,0.5){Large deviation};
% \end{tikzpicture}
% \caption{Overview of the results proven in this paper}
% \label{1}
% \end{figure}

\textbf{Organization:} 
The paper has the following structure. In \Cref{se2}, we give some preliminary contents and the statement of the main results, large deviation principle. In \Cref{se44}, we discuss the proof strategy of the large deviation principle. In \Cref{se4}, we get the limit $H$ of the nonlinear operator sequence $H_n$ as $n\to \infty$ about uniform topology, and consider the principle-eigenvalue problem. \Cref{se5} is devoted to the exponential tightness on non-compact space. Based on the nonlinear Hamilton-Jacobi equations and associated Hamilton-Jacobi-Bellman equations, we derive the comparison principle in \Cref{se6}. In \Cref{se7}, we give the proof of action-integral representation. In \Cref{se9}, we verify the existence and uniqueness of the CIR process with finite fast state-dependent switching. 

%--------Preliminaries--------------------------- 
\section{General setting} \label{se2}

To facilitate the presentation, we introduce some notation and definitions that will be used in later sections. Throughout the paper, we recall the sets of
$E:=(0,\infty)$, and $S=\{1,2,\ldots,N\}$, $N<\infty$ is a finite state space.  
The set of $C_b(E)$ is continuous and bounded functions, and $C^\infty_{cc}(E)$ is the set of smooth functions that are constant outside of a compact set. $\mathcal{D}_E(
\mathbb{R}^+)$ is the Skorokhod space of trajectories that are right-continuous and have left limits, equipped with Skorokhod topology, cf. \cite[Section 3.5]{EK2005}.

\subsection{Preliminaries}
Next, we introduce the basic definitions about large deviations.
 \begin{definition}\label{def_LDP}
	Let $\{X_n\}_{n\geq 1}$ be a sequence of random variables on Polish space $\mathcal{X}$. Furthermore, consider a function $I:\mathcal{X}\rightarrow[0,\infty]$. We say that
	\begin{itemize}
		\item the function $I$ is a \textit{rate function} if the set $\{x\in\mathcal{X}\,|\,I(x)\leq c\}$ is closed for every $c\geq 0$. The function $I$ is a  \textit{good rate function} if the set $\{x\in\mathcal{X}\,|\,I(x)\leq c\}$ is compact for every $c\geq 0$.
		\item the sequence $\{X_n\}_{n\geq 1}$ is \textit{exponentially tight} at speed $n$,  if for every $a\geq 0$, there exists a compact set $K_a\subseteq\mathcal{X}$ such that 
		\begin{equation*}
			\limsup_{n\to\infty}\frac{1}{n}\log \mathbb{P}(X_n\notin K^c_a)\leq -a.
		\end{equation*}
		\item the sequence $\{X_n\}_{n\geq1}$ satisfies the \textit{large deviation principle} with speed $n$ and good rate function $I$ if for every closed set $F\subseteq\mathcal{X}$, we have
		\begin{equation*}
			\limsup_{n\to\infty}\frac{1}{n}\log\mathbb{P}(X_n\in F)\leq-\inf_{x\in F}I(x),
		\end{equation*}
		and, for every open set $U \subseteq\mathcal{X}$, we have
		\begin{equation*}
			\liminf_{n\to\infty}\frac{1}{n}\log\mathbb{P}(X_n\in U)\geq-\inf_{x\in U}I(x).
		\end{equation*}
	\end{itemize}
\end{definition}

\begin{definition}[Absolutely continuous]
	We denote by $\mathcal{AC}(E)$ the space of absolutely continuous curves in $E$. 
	A curve $\gamma:[0,T]\to E$ is absolutely continuous if there exists a function $g\in L^1[0,T]$ such that for $t\in[0,T]$ we have $\gamma(t)=\gamma(0)+\int^t_0g(s)\text{d}s$. We write $g=\dot{\gamma}$.
	
	A curve $\gamma:[0,\infty)\to E$ is absolutely continuous, i.e. $\gamma\in \mathcal{AC}(E)$, if the restriction to $[0,T]$ is absolutely continuous for every $T>0$.
\end{definition}
\begin{definition}[Action-integral representation of rate function]
	We say that a rate function $I:\mathcal{D}_E(
	\mathbb{R}^+)\to [0,\infty]$ is of action-integral form if there is a non-trivial convex map $\mathcal{L}:E\times E\to[0,\infty]$ with which
	\begin{equation*}
		I(x)=
		\begin{cases}
			I_0(x(0))+\int^\infty_0\mathcal{L}(x(t),\dot{x}(t))\mathrm{d}t ,&\mbox{if~$x\in \mathcal{AC}(E)$},\\
			\infty,&\mbox{otherwise,}
		\end{cases}
	\end{equation*}
	where $I_0:E\to[0,\infty]$ is a rate function. We refer to the map $\mathcal{L}$ as the Lagrangian, i.e. $v \longmapsto \mathcal{L}(x,v)$ is convex and $(x,v) \longmapsto \mathcal{L}(x,v)$ is lower semicontinuous.
\end{definition}

%----------Main results-----------------------------

\subsection{CIR processes with finite state-dependent fast switching}\label{se3}

In this paper, we study CIR processes with fast switching on $E\times S$
\begin{equation}\label{eqn_CIR}
	\left\{
	\begin{array}{ll}
		\mathrm{d} X^\varepsilon_n(t)=\eta(\mu(\Lambda^\varepsilon_n(t))-X^\varepsilon_n(t))\text{d} t+n^{-\frac{1}{2}}\theta\sqrt{X^\varepsilon_n(t)}\mathrm{d} W(t),\\
		~~~~~~~~\\
		(X^{\varepsilon}_n(0),\Lambda^{\varepsilon}_n(0))=(x_0,k_0)\in
		E\times S,
	\end{array}
	\right.
\end{equation}
where the fast process $\Lambda^{\varepsilon}_n(t)$  is a jumping-process on $S$ satisfying
\begin{equation}\label{eqn_fasting_switch}
	\mathbb{P} (\Lambda^{\varepsilon}_n(t+\triangle)=j~|~\Lambda^{\varepsilon}_n(t)=i, X^{\varepsilon}_n(t)= x)=
	\begin{cases}
		\frac{1}{\varepsilon}q_{ij}(x)\triangle+ \circ(\triangle) ,&\mbox{if $j\neq i$,}\\
		1+\frac{1}{\varepsilon}q_{ij}(x)\triangle +\circ(\triangle),&\mbox{if $j=i$,}
	\end{cases}
\end{equation}
for $\Delta >0$, $i,j \in S$, $x\in E$, and $\varepsilon>0$ is a small parameter.  
The system $(X^{\varepsilon}_n(t),\Lambda^{\varepsilon}_n(t))$ is a Markov process.
\par
The slow process $X^\varepsilon_n(t)$ is mean-reverting singular diffusion.
$X^\varepsilon_n(t)$ stands for the instantaneous interest rate at time $t$;
$\eta>0$ is the rate of mean reversion;
for any $i\in S$, $\mu(i)$ represents the mean of the interest rate;
$\theta>0$ is the standard deviation of the interest rate.
$(W(t))_{t\ge0}$ is a real value Brownian motion defined on the probability space $(\Omega,\mathscr{F},\mathbb{P})$ with the filtration $(\mathscr{F}_t)_{t\ge0}$ satisfying the usual condition (i.e., $\mathscr{F}_0$ contains all $\mathbb{P}$-null sets and $\mathscr{F}_t=\mathscr{F}_{t+}:=
\bigcap_{s>t}\mathscr{F}_s$). 
\par
The fast processes $\Lambda^\varepsilon_n(t)$ is a finite  state-dependent switching process. In particular, if $S=\{1\}$, \eqref{eqn_CIR} is often used to characterize the interest rate in finance which is called classical CIR processes without switching.   
 \par 
 Before studying the process $(X^{\varepsilon}_n(t),\Lambda^{\varepsilon}_n(t))$, we first give a result on the existence and uniqueness of the process.
\begin{proposition}[Existence and uniqueness]\label{pro_exist_uniq} For every $i$, assume that $2\eta \mu(i)\geq \theta^2$, then
	the systems \eqref{eqn_CIR} and \eqref{eqn_fasting_switch} have a nonnegative unique strong solution $(X^{\varepsilon}_n(t),\Lambda^{\varepsilon}_n(t))$ with initial value $(X^{\varepsilon}_n(0),\Lambda^{\varepsilon}_n(0))=(x_0,k_0)$, and $(X^{\varepsilon}_n(t),\Lambda^{\varepsilon}_n(t))$ is non-explosive.
\end{proposition}
The proof of \Cref{pro_exist_uniq} is deferred to \Cref{se9}.

\subsection{Main results}
Before giving a main result, we set the assumptions that will be necessary for the main result. 
\begin{assumption}\label{assu_varn}
	Let $\varepsilon=\frac{1}{n}$, this shows that small disturbance and fast switching have the same rate.
\end{assumption}
\begin{assumption}\label{asm_conservative}
	For any $x\in E$, $(q_{ij}(x))_{i,j\in S}$ is a conservative, irreducible transition rate matrix, and $\sup_{i\in S}q_i(x)<\infty$, where $q_i(x)=-q_{ii}(x)=\sum_{j\in S,j\neq i}q_{ij}(x)$.
\end{assumption}
\begin{assumption}\label{asm_q_conti}
	There exists a constant $C>0$ such that 
	\begin{equation*}
		|q_{ij}(x)-q_{ij}(y)|\leq C|x-y|,~~x,y\in E,~i,j \in S.
	\end{equation*}
\end{assumption}
\begin{remark}
    If \Cref{assu_varn} is satisfied, \eqref{eqn_CIR} and \eqref{eqn_fasting_switch} become \begin{equation}\label{eqn_new_CIR}
	\left\{
	\begin{array}{ll}
		\text{d} X_n(t)=\eta(\mu(\Lambda_n(t))-X_n(t))\text{d} t+n^{-\frac{1}{2}}\theta\sqrt{X_n(t)}\text{d} W(t),\\
		~~~~~~~~\\
		(X_n(0),\Lambda_n(0))=(x_0,k_0)\in
		E\times S,
	\end{array}
	\right.
\end{equation}
and
\begin{equation}\label{eqn_new_fasting_switch}
	\mathbb{P} (\Lambda_n(t+\triangle)=j~|~\Lambda_n(t)=i, X_n(t)= x)=
	\begin{cases}
		nq_{ij}(x)\triangle+ \circ(\triangle),&\mbox{if $j\neq i$,}\\
		1+nq_{ij}(x)\triangle +\circ(\triangle),&\mbox{if $j=i$.}
	\end{cases}
\end{equation}
 From now on, except for the \Cref{se9} we use $(X^\varepsilon_n(t),\Lambda^\varepsilon_n(t))$ instead of $(X_n(t),\Lambda_n(t))$. \Cref{asm_conservative} guarantees that there exists a unique stationary distribution $\pi^x(t)=(\pi^x_i(t))_{i\in S}$ for the fast process $\Lambda_n(t)$ if slow process is fixed at $x$. Moreover, the system \eqref{eqn_new_CIR} and \eqref{eqn_new_fasting_switch} will exhibit the averaging principle in the law of large numbers limit and we identify the limit equation as
 \begin{equation*}
     \mathrm{d}\bar{X}(t)=\eta \bigg(\sum_{i\in S}\mu(i)\pi^x_i(t)-\bar{X}(t)\bigg)\mathrm{d}t.
 \end{equation*}
 We need \Cref{asm_q_conti} to prove comparison principle for technical reasons.
\end{remark}
Here, we state the path large deviation principles of the Markov process $(X_n(t),\Lambda_n(t))$, which is the main result in our paper.
\begin{theorem}[Large deviation principles]\label{thm_LDP} 
Let $(X_n(t),\Lambda_n(t))$ be the Markov processes \eqref{eqn_new_CIR} and \eqref{eqn_new_fasting_switch} on $E\times S$.
	Suppose that 
	\begin{itemize}
	\item the large deviation principle holds for $X_n(0)$ on $E$ with speed $n$ and a good rate function $I_0$;
	\item Assumptions \ref{assu_varn}, \ref{asm_conservative} and \ref{asm_q_conti} are satisfied.
	\end{itemize}

Then, the large deviation principle with speed $n$ holds for $X_n(t)$ on $\mathcal{D}_E(
\mathbb{R}^+)$ with a good rate function $I$ having action-integral representation,
\begin{equation*}
	I(\gamma)=
	\begin{cases}
		I_0(\gamma(0))+\int^\infty_0\mathcal{L}\left(\gamma(s),\dot{\gamma}(s)\right)\mathrm{d}s,&\mbox{if~$\gamma\in\mathcal{AC}(E)$,}\\
		\infty,&\mbox{otherwise}
	\end{cases}
\end{equation*}
with $ \mathcal{L}(x,v):=\sup_{p\in\mathbb{R}}\{\langle p,v\rangle -\mathcal{H}(x,p)\}$ which is the Legendre dual  of $\mathcal{H}$ given by 
\begin{equation}\label{eqn_1mathcalH} \mathcal{H}(x,p)=\sup_{\pi\in\mathcal{P}(S)}\big\{\int_E B_{x,p}(z)\pi(\text{d}z)-\mathcal{I}(x,\pi)\big\},
\end{equation} 
where 
\begin{equation*}
	B_{x,p}(i)=\eta(\mu(i)-x)p+\frac{1}{2}\theta^2xp^2 
\end{equation*}
coming from the slow process $X_n(t)$ and Donsker–Varadhan function
\begin{equation*}
	\mathcal{I}(x,\pi)=-\inf_{g>0}\int_E\frac{R_xg(z)}{g(z)}\pi(\text{d}z),
\end{equation*}
where $R_x$ is the generator corresponding to the fast process $\Lambda_n(t)$ defined by 
\begin{equation*}
	R_xg(z)=\sum_{j\in S }q_{zj}(x)\left(g(j)-g(z)\right).
\end{equation*}	
 \end{theorem}

%----------Strategy of proof of--------------
	
\section{Strategy of proof of \Cref{thm_LDP}}\label{se44}
The verification of \Cref{thm_LDP} is based on \Cref{proposition_main} below which reduces proving LDP to check
\begin{itemize}
    \item exponential tightness;
    \item convergence of operators;
    \item well-posedness of a Hamilton-Jacobi equation in terms of the limiting operator.
\end{itemize}
We will formulate well posedness for Hamilton-Jacobi equations in terms of viscosity solutions, and we begin with the definition of viscosity solutions.
\subsection{Viscosity solutions}
\begin{definition}[Viscosity solutions] \label{def_VS}Let $H\subseteq C_b(E)\times C_b(E\times S)$ be a multivalued operator. We denote $\mathcal{D}(H)$ for the domain of $H$ and $\mathcal{R}(H)$ for the range of $H$. Let $\lambda>0$ and $h\in C_b(E)$. Consider the Hamilton-Jacobi equation
	\begin{equation}\label{HJ}
		f-\lambda H f=h.
	\end{equation}
	\begin{description}
		\item [Classical solutions] 
		We say that $u$ is a classical subsolution of \eqref{HJ} if there is a function $g$ such that $(u,g)\in H$ and $u - \lambda g \leq h$. We say that $v$ is a classical supersolution of \eqref{HJ} if there is a function $g$ such that $(v, g) \in H$ and $v - \lambda g \geq h$. We say that $u$ is a classical solution if it is both a subsolution and a supersolution.
		\item[Viscosity subsolutions] We say that $u$ is a (viscosity) subsolution of \eqref{HJ} if $u$ is bounded, upper semicontinuous, and if for every $(f,g)\in H$ there exists a sequence $(x_n,z_n) \in E\times S$ such that
		\begin{equation*}
			\lim_{n\to\infty}u(x_n)-f(x_n)=\sup_{x}u(x)-f(x),
		\end{equation*}
		\begin{equation*}
			\limsup_{n\to\infty}u(x_n)-\lambda g(x_n,z_n)-h(x_n)\leq 0.
		\end{equation*}
		\item[Viscosity supersolutions] We say that $v$ is a (viscosity) supersolution of \eqref{HJ} 
		if $v$ is bounded, lower semicontinuous, and if for every $(f,g)\in H$ there exists a sequence sequence $(x_n,z_n) \in E\times S$ such that
		\begin{equation*}
			\lim_{n\to\infty}v(x_n)-f(x_n)=\sup_{x}v(x)-f(x),
		\end{equation*}
		\begin{equation*}
	\liminf_{n\to\infty}v(x_n)-\lambda g(x_n,z_n)-h(x_n)\geq 0.
		\end{equation*}
		\item [Viscosity solutions] We say that $u$ is a (viscosity) solution of \eqref{HJ} if it is both a subsolution and a supersolution to \eqref{HJ}. 
	\end{description}
\end{definition}
\begin{remark}\label{rem_compact}
	Consider the definition of subsolutions. Suppose that the test function $(f,g)\in H$ has compact sublevel sets, then instead of working with a sequence $(x_n,z_n)$, we can pick $(x_0,z_0)$ such that
	\begin{equation*}
		u(x_0)-f(x_0)=\sup_{x}u(x)-f(x),
	\end{equation*}
	\begin{equation*}
		u(x_0)-\lambda g(x_0,z_0)-h(x_0)\leq 0.
	\end{equation*}
	Similarly, a simplification holds in the case of supersolutions. This is used in the proof \Cref{lem_1d} below.
\end{remark}
\begin{definition}[Comparison principle]
	We say that \eqref{HJ} satisfies the comparison principle if for every viscosity subsolutions $u$ and viscosity supersolutions $v$ to \eqref{HJ}, we have $u\leq v$.
\end{definition}

\begin{remark}[Uniqueness] 
	The comparison principle implies uniqueness of viscosity solutions. Suppose that $u$ and $v$ are both viscosity solutions, then the comparison principle yields that $u\leq v$ and $v\leq u$, implying that $u=v$.
\end{remark}

\subsection{Reduction of LDP}
We proceed to introduce the object that plays a crucial role in \Cref{proposition_main} below.
\par
For each $f\in C_b(E\times S)$, we define the conditional
log-Laplace transforms \begin{equation}\label{eqn_Vn}
V_n(t)f(x_0,k_0)=\frac{1}{n}\text{\rm{log}}\mathbb{E}[\mathrm{e}^{nf(X_n (t),\Lambda_n(t))}~|~(X_n(0),\Lambda_n(0))=(x_0,k_0)].
\end{equation}
From the above construction, we see that $V_n(t)$ is a nonlinear semigroup. Calculating the generator $H_n$ of the semigroup $V_n(t)$, we formally find by the chain rule that 
\begin{equation}\label{eqn_semigroup_and_generator}
H_nf:=\frac{\mathrm{d}}{\mathrm{d}t}V_n(t)f \Big|_{t=0} =\frac{1}{n}\mathrm{e}^{-nf}A_n\mathrm{e}^{nf},
\end{equation}
which for the generator $A_n$ corresponding to \eqref{eqn_new_CIR} and \eqref{eqn_new_fasting_switch} leads to
\begin{equation}\label{eqn_Hn}
\begin{split}
H_nf(x,i)&=\eta\left(\mu(i)-x\right)\partial_{x} f(x,i)+\frac{1}{2}\theta^2x \left(\partial_{x}f(x,i)\right)^2+\frac{1}{2n}\theta^2x\partial_{xx}f(x,i)\\
&\quad+\sum_{j\in S}q_{ij}(x)\big(\mathrm{e}^{n(f(x,j)-f(x,i))}-1\big).
\end{split}
\end{equation}
\par
Furthermore, recalling the definition of $\mathcal{L}$ in \Cref{thm_LDP}, we give the definition of Nisio semigroup and resolvent $\mathbf{R}(\lambda)$, which features in \Cref{proposition_main} below.
	\begin{definition}[Nisio semigroup]\label{def_Nisio__semigroup}
		Define the Nisio semigroup for measurable functions $f$ on $E$:
\begin{equation}\label{eqn:Nisio_semigroup}
			\mathbf{V}(t)f(x)=\sup_{\gamma\in\mathcal{AC},\atop \gamma(0)=x}\big\{f(\gamma(t))-\int^t_0\mathcal{L}(\gamma(s),\dot{\gamma}(s))\text{d}s\big\}.
		\end{equation}
	
%The operator $\mathbf{H}$ is based on the variational $\mathcal{H}$.
	\end{definition}
	\begin{definition}[Resolvent $\mathbf{R}(\lambda)$]
	For $\lambda>0$ and $h\in C_b(E)$, define the resolvent $\mathbf{R}(\lambda)h: E\to \mathbb{R}$ by
	\begin{equation}\label{eqn_Rlambda}
		\begin{split}
			\mathbf{R}(\lambda)h(x)=\sup_{\gamma\in\mathcal{AC}\atop\gamma(0)=x}\int^\infty_0\lambda^{-1}\mathrm{e}^{-\lambda^{-1}t}\bigg(h(\gamma(t))-\int^t_0\mathcal{L}(\gamma(s),\dot{\gamma}(s))\mathrm{d}s\bigg)\mathrm{d}t.
		\end{split}
	\end{equation}
\end{definition}
Now, we are ready to provide the main proposition for proving \Cref{thm_LDP}.
		\begin{proposition}[Adaptation of Theorem 5.15, Theorem 8.27 and Corollary 8.28 in \cite{FK2006} to our context]\label{proposition_main}
		Let $(X_n(t),\Lambda_n(t))$ be Markov processes on $E\times S$.
		Suppose that
		\begin{enumerate}
			\item  \label{item_a} $X_n(0)$  satisfies large deviation principle;
			\item \label{item_b} there exists an operator $H$ such that we have $\|H_n-H\|\to 0$ as $n\to\infty$ in the sense of \Cref{pro_limit operator};
			\item \label{item_c} we have exponential tightness of the process $(X_n(t),\Lambda_n(t))$;
			\item  \label{item_d} for all $\lambda>0$ and $h\in C_b(E)$, the comparison principle holds for $f-\lambda Hf=h$ with same $H$ in \cref{item_b}, and there exists a unique solution $R(\lambda)h$;
  \item \label{item_e} $R(\lambda)h=\mathbf{R}(\lambda)h$.
		\end{enumerate}
	 
		Then the following hold:
		\begin{enumerate}[(i)]
			\item \label{item_Limit of nonlinear semigroup} (Limit of nonlinear semigroup) For each $x\in E$, $f\in \overline{\mathcal{D}(H)}$, $f_n\in \mathcal{B}(E\times S)$ and $t\geq 0$, if $\|f_n-f\|\to 0$ as $n\to \infty$, there exists  a unique semigroup $V(t)$ such that  
\begin{equation}\label{nonlinear}
		\lim_{n\to\infty}\|V_n(t)f_n-V(t)f\|=0
			\end{equation}
		and 
  \begin{equation}\label{eqn_RV}
  \lim_{m\to\infty}\|R(t/m)^m f-V(t)f\|=0.
  \end{equation}
			\item \label{item_Large deviation principle}(Large deviation principle) $X_n(t)$ satisfies the large deviation principle with good rate function $I$ given by
			\begin{equation}\label{eqn_rate_function_1}
			I(x)=I_0(x(t_0))+\sup_{k\in \mathbb{N}}\sup_{0=t_0<t_1<\cdots<t_k<\infty}\sum_{i=0}^{k}I^V_{t_{i+1}-t_{i}}(x(t_{i+1})|x(t_{i})),
			\end{equation}
			where for $\Delta t=t_{i+1}-t_{i}>0$ and $x(t_{i+1}),x(t_{i})\in E$, the conditional rate functions $I^V_{\Delta t}(x(t_{i+1})~ |~ x(t_{i}))$ are 
			\begin{equation*}
				I^V_{\Delta t}(x(t_{i+1})~|~x(t_{i}))=\sup_{f\in C_b(E)}[f(x(t_{i+1}))-V(\Delta t)f(x(t_{i}))].
			\end{equation*} 
			\item \label{item_Action-integral representation of the rate function} (Action-integral representation of the rate function) 
  The rate function $I$ of \eqref{eqn_rate_function_1} together with $V(t)=\mathbf{V}(t)$ is also  given in Action-integral representation,
\begin{equation}\label{eqn_AC}
I(\gamma)=
\begin{cases}
I_0(\gamma(0))+\int^\infty_0\mathcal{L}(\gamma(s),\dot{\gamma}(s) )\text{d}s,&\mbox{if~$\gamma\in \mathcal{AC}(E)$,}\\
					\infty,&\mbox{otherwise.}
				\end{cases}
			\end{equation}
		\end{enumerate}
	\end{proposition}

	\subsection{Outline of the proof of \Cref{proposition_main}}\label{sub_outline}
In this subsection, we sketch why \Cref{proposition_main} is true before using it to give a full proof of \Cref{thm_LDP}. That is, we check the assumptions (a)-(e) of \Cref{proposition_main} in detail in Sections 4-7.
\begin{description}
	\item[Step 1:] 
	For the Markov process $(X_n(t), \Lambda_n(t))$, we aim to prove LDP on the path space. The approach is based on a variant of the projective limit theorem,  \Cref{lem_428} as below. Namely, if a sequence of the processes is exponentially tight in the Skorokhod space, then it suffices to establish LDP of finite-dimensional distributions. Moreover, the rate function is given in the projective limit form: it is given as the supremum over the rate functions of the finite-dimensional distributions.
	\item[Step 2:] We show the exponential tightness of $(X_n(t), \Lambda_n(t))$ and the LDP for the finite dimensional distributions. 
	To do this, exponential tightness is \cref{item_c} of \Cref{proposition_main} and is proven in \Cref{se5}. We are left to prove LDP for finite-dimensional distributions, which is established via Bryc's theorem, \Cref{lem_325} as below. For which one needs to prove convergence of log expectations.
	\par
	For simplicity, we consider the log expectation for $k=1$ with $f_0$, $f_1\in C_b(E\times S)$ and $0=t_0<t_1$ only, and have
\begin{equation}\label{eqn_Gamman}
	\begin{split}	\Gamma_n&=\Gamma_n(f_0,f_1)\\
 &:=\frac{1}{n}\log\mathbb{E}\big(\mathrm{e}^{n f_0(X_n(t_0),\Lambda_n(t_0))+n f_1(X_n(t_1),\Lambda_n(t_1))}\big)\\
		&=\frac{1}{n}\log \mathbb{E}[\mathbb{E}(\mathrm{e}^{n f_0(X_n(t_0),\Lambda_n(t_0))+n f_1(X_n(t_1),\Lambda_n(t_1))}\Large |(X_n(t_0),\Lambda_n(t_0))=(x_0,k_0))]\\
		&=\frac{1}{n}\log \mathbb{E}\big(\mathrm{e}^{n f_0(x_0,k_0)+n V_n(t_1) f_1(x_0,k_0)} \big),
	\end{split}
 \end{equation}
where in the third equality we used Markov property and the conditional log-Laplace transform $V_n(t_1)$ is defined in \eqref{eqn_Vn}. Moreover, \eqref{eqn_Gamman} reduces to proving
\begin{enumerate}[a)]
    \item \label{item_a2} the LDP for $X_n(t_0)$  with rate function $I_0$, which is \cref{item_a} of \Cref{proposition_main};
    \item \label{item_b2}  the convergence of the conditional log-Laplace transform $V_n(t_1)$. 
\end{enumerate}
To proceed, we defer the proof of \cref{item_b2} to Step 3. Given that \cref{item_b2} is true, there exists a limit $\Gamma$ of $\Gamma_n$ as $n\to \infty$. We note that the limit is $f_0(z)+V(t_1)f_1(z)$. Combining this limit, we first use \Cref{lem_varadhan} and take over sup about $f_0$ and $f_1$ via \Cref{lem_325} in the first equality below. The rate function for $X_n(t_0)$, $X_n(t_1)$ by Bryc's theorem is given by
 \begin{align*}
     I_{t_0,t_1}&(x(t_0),x(t_1))\\&=\underbrace{\sup_{f_0,f_1}\bigg(f_0(x(t_0))+f_1(x(t_1))-\underbrace{\sup_{z}\big(f_0(z)+V(t_1)f_1(z)-I_0(z)\big)}_{\text{Lem}~\ref{lem_varadhan}}\bigg)}_{\text{Lem}~\ref{lem_325}}\\
     &=\sup_{g_0,f_1}\left(g_0(x(t_0))-V(t_1)f_1(x(t_0))+f_1(x(t_1))-\sup_{z}\big(g_0(z)-I_0(z)\big)\right)\\
     %&=\sup_f\inf_{x_0}(f(x_0)+V(t_1)f(x_0)+f(x_0)-f(x_0)-V(t_1)f(x_0))+\sup_f(f(x(t_1)-V(t_1)f(x_0))\\
     &=\sup_{g_0}\inf_{z}\big(g_0(x(t_0))-g_0(z)+I_0(z)\big)+\sup_{f_1}\big(f_1(x(t_1)-V(t_1)f_1(x(t_0))\big)\\
     &=:I_0(x(t_0))+I^V_{\Delta t}(x(t_1)|x(t_0)),
 \end{align*}
	where in the second equality we define $g_0:=f_0+V(t_1)f_1$ for simplicity, and in the last equality we have
	\begin{equation*}
I_0(x(t_0))=\sup_{g_0}\inf_{z}\big(g_0(x(t_0))-g_0(z)+I_0(z)\big)
	\end{equation*}
	and
	\begin{equation*}
		I^V_{\Delta t}(x(t_1)|x(t_0))=\sup_{f_1}\big(f_1(x(t_1)-V(t_1)f_1(x(t_0))\big),~~~\Delta t=t_1-t_0.
	\end{equation*}
By induction, we get LDP for the finite-dimensional distributions of $X_n(0)$, $X_n(t_1)$,$\ldots$, $X_n(t_k)$ with rate function
\begin{equation*}
    I^V_{t_0\ldots t_k}(x(t_0),\ldots,x(t_k))=I^V_0(x(t_0))+\sum_{i=0}^k I^V_{t_{i+1}-t_i}(x(t_{i+1})|x(t_i)).
\end{equation*}
\item[Step 3:]
 We are left to establish \ref{item_b2}, $\|V_n(t)f-V(t)f\|\to 0$ for any $t\geq 0$ and $n\to \infty$, assumed in step 2. 
 To do this, we achieve the goal by the Trotter-Kato-Kurtz theorem, \Cref{lem_55} as below. From the theorem, we need to check the following conditions for any $f\in C_b(E)$:
 \begin{enumerate}[(i)]
     \item \label{item_1} $H_n$ is the generator of semigroup $V_n$;
     \item \label{item_2} $H$ is the generator of semigroup $V$;
     \item \label{item_3}     
     $\lim_{n\to \infty}\|V_n(t)f-(\bONE-\frac{t}{n}  H_n)^{-n}f\|=0$;
      \item \label{item_4}
       $\lim_{n\to\infty}\|V(t)f-(\bONE-\frac{t}{n} H)^{-n}f\|=0$;
     \item \label{item_5} $\lim_{n\to \infty}\|H_n-H\|= 0$.
 \end{enumerate}
 The statement of \cref{item_1} and \cref{item_3} is similar to \cref{item_2} and \cref{item_4}, but the proof is completely different.
 \Cref{item_1}, which is proven in \eqref{eqn_semigroup_and_generator}.
For \cref{item_3}, we obtain it by the semigroup generation theorem.
 \par
We cite the Crandall-Liggett theorem, \Cref{lem_CL} as below, to show \cref{item_2} and \cref{item_4} by modifying \cref{item_5}. Two conditions need to be verified, dissipativity and the range condition. The ﬁrst one, the dissipativity of $H$ holds since the operator $H_n$ are dissipative. 

The second one is the range condition: for sufficiently many $h\in C_b(E)$ and all $\lambda>0$ one can find an $f\in\mathcal{D}(H)$ that solves the equation $f-\lambda Hf=h$ in the classical sense. Moreover, if $H$ is dissipative, then such solution is unique.

However, for non-linear equations, the verification of the range condition is very hard and it was observed early on \cite{CL1983} that viscosity solutions can be used to replace classical solutions. By weakening the type of solution needed for $(\bONE-\lambda H)f=h$, we have to require a strong form of uniqueness condition known as the comparison principle. Informally, this principle states that, if upper semicontinuous $\bar{f}$ and lower semicontinuous $\underline{f}$ satisfy
\begin{equation}\label{eqn_barf}
    (\bONE-\lambda H)\bar{f} \leq h ~~\mbox{and}~~(\bONE-\lambda H)\underline{f}\geq h,
\end{equation}
then $\bar{f} \leq \underline{f}$. The $\bar{f}$ and $\underline{f}$ are called, respectively, a viscosity subsolution and a viscosity supersolution, and are not necessarily in the domain of $H$.

%Note that the above formulation is only based on inequalities for sub- and supersolutions.
This provides an opportunity for further relaxation of conditions. If \eqref{eqn_barf} holds, we can introduce two more operators: $H_0$, $H_1$ such that $Hf \leq H_0f$ and $Hf \geq H_1f$ for all $f \in \mathcal{D}(H) \cap \mathcal{D}(H_i)$. Then
\begin{equation*}
    (\bONE-\lambda H_0)\bar{f} \leq h~~\mbox{and}~~(\bONE-\lambda H_1)\underline{f} \geq h.
\end{equation*}
Later on, formulate $H_0$, $H_1$ in terms of a Lyapunov function to restrict further analysis to compact sets. It suffices to establish comparison principle for $H_0$, $H_1$ in the sense of viscosity solutions.
 % Suppose that the comparison principle still holds for the above two “in-equations” (i.e., $\bar{f} \leq \underline{f}$ ). 

Next we turn to the existence of viscosity solutions. For this we use the Barles-Perthame procedure. The construction of $\underline{f}$, $\bar{f}$ by the Barles–Perthame procedure then reveals that $\bar{f}=\underline{f}= f \in C_b(E)$. Hence, each h uniquely corresponds to an $f \in C_b(E)$, and we can denote it by $f = R(\lambda) h$. Consequently, at least formally, $R(\lambda)= (\bONE -\lambda H)^{-1}$. In other words, $H_0$, $H_1$ implicitly determine $H$ through its resolvent. We can now completely avoid using \cref{item_5} $\lim_{n\to\infty}\|H_n- H\|= 0$, and we replace it by: for each
\begin{equation*}
    H_1f\leq \liminf_{n\to \infty}H_nf_n,~~~\limsup_{n\to \infty}H_nf_n\leq H_0f~~~~\mbox{some}~~~ f_n\to f,
\end{equation*}
in the strongly uniformly sense.
\par
From the above analysis, the conditions \cref{item_1}-\cref{item_5} reduce to establishing that viscosity solutions of $f-\lambda H_nf=h$ converge to viscosity solutions of $f-\lambda Hf=h$. 
 
To do this, by the Barles-Perthame procedure, for every $x\in E$, there exists a sequence $x_n\in E$ such that $\lim_{n\to \infty} x_n=x$, one can show that 
\begin{equation*}
    u(x):=\sup\bigg\{\limsup_{n\to\infty}R_n(\lambda)h(x_n) ~|~ \lim_{n\to \infty}x_{n} = x\bigg\},
\end{equation*}
\begin{equation*}
   v(x):=\inf\bigg\{\liminf_{n\to\infty}R_n(\lambda)h(x_n)~|~ \lim_{n\to \infty}x_{n} = x \bigg\},
\end{equation*} 
are a viscosity subsolution and a viscosity supersolution to $f-\lambda Hf=h$. It is obvious that $ u\geq v$ from the construction. In addition, from \cref{item_d} of \Cref{proposition_main}: the comparison principle is satisfied, we obtain that $u=v=R(\lambda)h$ which is the unique viscosity solution. 
Next, using this solution, we can extend the domain of the operator $H$ by adding all pairs of the form $(R(\lambda)h,\lambda^{-1}(R(\lambda)h-h))$ to the graph of $H$ to obtain a new operator $\hat{H}$:
\begin{equation*}
		\hat{H}=\left\{\left(R(\lambda)h,\lambda^{-1}(R(\lambda)h-h)\right)\,\middle|\, \lambda>0,h\in C_b(X)\right\}.
	\end{equation*}
Furthermore, we prove this extension operator $\hat{H}$ satisfying the conditions of the Crandall-Liggett theorem.
Firstly, the range condition holds by construction. Secondly, $\hat{H}$ is a dissipative operator as it satisfies the positive maximum principle, \Cref{lem_positive_maximum_principle} as below, and \cite{K2022} have proven it on proposition 4.10.

Using the Crandall-Liggett theorem once again, $\hat{H}$ generates a semigroup $V(t)$. Subsequently, thanks to the Trotter-Kato-Kurtz theorem, we achieve the goal $\|V_nf(t)-Vf(t)\|\to 0$ for any $t\geq 0$, $f\in \overline{\mathcal{D}(\hat{H})}$, as $n\to\infty$.
%hat{H} 满足CL条件，那么怎么过渡到H满足CL条件呢？
\item[Step 4:] 
 From steps 1, 2, and 3, we obtain that the LDP is satisfied and with a rate function in the projective limit form \eqref{eqn_rate_function_1}.
However, the rate function is only implicitly characterized by $V(t)$. This is why we establish a Lagrangian form rate function \eqref{eqn_AC} based on \eqref{eqn_rate_function_1}. To do so it is sufficient to prove that for any $f\in C_b(E)$,
	\begin{equation}\label{eqn_VVNS}
		V(t)f=\mathbf{V}(t)f
	\end{equation}
	by the method of resolvent approximation, where $\mathbf{V}(t)$ is defined in \eqref{eqn:Nisio_semigroup}.
 
 To do this, we connect the variational semigroup to the resolvent. We first recall the Nisio semigroup of \eqref{eqn_Rlambda}:
\begin{equation*}
		\begin{split}
			\mathbf{R}(\lambda)h(x)=\sup_{\gamma\in\mathcal{AC}\atop\gamma(0)=x}\int^\infty_0\lambda^{-1}\mathrm{e}^{-\lambda^{-1}t}\big(h(\gamma(t))-\int^t_0\mathcal{L}(\gamma(s),\dot{\gamma}(s))\mathrm{d}s\big)\mathrm{d}t.
		\end{split}
	\end{equation*}
 If conditions 8.9, 8.10, and 8.11 in \cite{FK2006} hold,  we obtain the following important results.
 \begin{itemize}
     \item For $\lambda>0$, we have $\mathbf{R}(\lambda)C_b(E)\subseteq C_b(E)$ and $\mathbf{R}(\lambda)h$ is a viscosity solution of $f-\lambda \mathbf{H}f=h$, following the ﬁrst part of the proof of \Cref{lem_827};
     \item for $f \in C_b(E)$,
     \begin{equation}\label{eqn_VBFR}
  \lim_{n\to\infty}\|\mathbf{V}(t)f-\mathbf{R}(t/n)^{n}f\|=0, 
\end{equation}
thanks to \Cref{lem_818}.
 \end{itemize}
\par
 In addition, combining the comparison principle, \cref{item_d} of \Cref{proposition_main}, with \Cref{3}, we get
\begin{equation}\label{eqn_RR}
\mathbf{R}(\lambda)h = R(\lambda)h,
\end{equation}
which is \cref{item_e} of \Cref{proposition_main}. 

Subsequently, from \eqref{eqn_RV}, \eqref{eqn_VBFR} and \eqref{eqn_RR}, we obtain
 $V (t)f = \mathbf{V}(t)f$. 
 
 Note that we postpone checking the conditions 8.9, 8.10, and 8.11 in \cite{FK2006}, and put it as a \Cref{lem_89810811} on \Cref{se7}.
\end{description}
\par
Next, we give the auxiliary lemmas used in Steps 1-4 in the order of occurrence.
\begin{lemma}[Projective limit theorem, Theorem 4.28 in \cite{FK2006}]\label{lem_428}
	Assume that $\{X_n\}$ is exponentially tight in $\mathcal{D}_{E}[0,\infty)$ and that for each $0\leq t_1< t_2<\cdots<t_m$, $\{(X_n(t_1),\ldots,X_n(t_m))\}$ satisfies the large deviation principle in $E^m$ with rate function $I_{t_1,\ldots,t_m}$. Then $\{X_n\}$ satisfies the large deviation principle in 
	$\mathcal{D}_{E}[0,\infty)$ with good rate function
	\begin{equation*}
		I(x)=\sup_{\{t_i\}\subset\Delta^c_x}I_{t_1,\ldots,t_m}((x(t_1),\ldots,x(t_m)),
	\end{equation*}
 where $\{t_i\}$ is shorthand for all sets of the form $\{t_1,t_2,\ldots,t_m\}$ and $\Delta_x$ is the set of times where $x$ is discontinuous.
\end{lemma}
\begin{lemma}[Bryc's theorem, Proposition 3.25 in \cite{FK2006}]\label{lem_325}
	Suppose $\{(X_n,Y_n)\}$ is exponentially tight in the product space $(S_1\times S_2,d_1+d_2)$. Let $\mu_n\in\mathcal{P}(S_1\times S_2)$ be the distribution of $(X_n,Y_n)$ and let $\mu_n(\mathrm{d}x\times \mathrm{d}y)=\eta_n(\mathrm{d}y|x)\mu^1_n(\mathrm{d}x)$, that is, $\mu^1_n$ is the $S_1$-marginal of $\mu_n$ and $\eta_n$ gives the conditional distribution of $Y_n$ given $X_n$. Suppose that for each $f\in C_b(S_2)$
	\begin{equation*}
		\Lambda_2(f|x)=\lim_{n\to\infty}\frac{1}{n}\log \int_{S_2}\mathrm{e}^{nf(y)}\eta_n(\mathrm{d}y|x)
	\end{equation*}
	exists, that the convergence is uniform for $x$ in compact subsets of $S_1$, and that $\Lambda_2(f|x)$ is a continuous function of $x$. For $x\in S_1$	and $y\in S_2$, define
	\begin{equation*}
		I_2(y|x)=\sup_{f\in C_b(S_2)}(f(y)-\Lambda_2(f|x)).
	\end{equation*}
	If $\{X_n\}$ satisfies the large deviation principle with good rate function $I_1$, then
	$\{(X_n,Y_n)\}$ satisfies the large deviation principle with good rate function
	\begin{equation*}
		I(x,y)=I_1(x)+I_2(y|x).
	\end{equation*}
\end{lemma}
\begin{lemma}[Varadhan's Lemma, Theorem III.13 in \cite{DH2008}]\label{lem_varadhan}
	Let $(P_n)$ satisfy the  LDP on $\mathcal{X}$ with rate $n$ and with rate function $I$. Let $F_n:\mathcal{X}\to \mathbb{R}$ be a continuous function that is bounded from above and $\|F_n- F\|\to 0$ when $n\to \infty$. Then
	\begin{equation*}
		\lim_{n\to\infty}\frac{1}{n}\log \int_{\mathcal{X}}\mathrm{e}^{n F_n(x)}       P_n(\mathrm{d}x)=\sup_{x\in\mathcal{X}}(F(x)-I(x)).
	\end{equation*}
\end{lemma}
The following theorem is a simplification of Proposition 5.5 in [FK06].
\begin{lemma}[Trotter-Kato-Kurtz theorem, Proposition 5.5 in \cite{FK2006}]\label{lem_55}
  Let $E$ be a Polish space and let $H_n:C_b(E)\to C_b(E)$ and $H:\mathcal{D}(H)\subseteq C_b(E)\to C_b(E)$ be dissipative operators that satisfy the range condition with the same $\lambda$. Let $V_n(t)$ and $V(t)$ be the corresponding generated semigroups in the Crandall-Liggett sense. Suppose that the following:
 \par 
 For each $f\in \mathcal{D}(H)$, there exist $f_n\in \mathcal{D}(H_n)$ such that
      \begin{equation*}
          \|f-f_n\|\overset{n\to \infty}\longrightarrow 0~~\mbox{and}~~\|Hf-H_nf_n\|\overset{n\to\infty}\longrightarrow 0.    \end{equation*}
\par
  Then for any $f\in \overline{\mathcal{D}(H)}$ and $f_n\in C_b(E)$ such that $\|f-f_n\|\to 0$, we have
  \begin{equation*}
      \|V(t)f-V_n(t)f_n\|\overset{n\to \infty}\longrightarrow 0.
    \end{equation*}
\end{lemma}

\begin{lemma}[Crandall-Liggett theorem in \cite{CL1971}]\label{lem_CL}
	Let $H$ be an operator on a Banach space $X$.  Suppose that
	\begin{enumerate}
		\item \label{item_CLa} $H$ is dissipative. We say $H\subseteq C_b(E)\times C_b(E\times S)$ is dissipative if for all $(f_1,g_1)$, $(f_2,g_2)\in H$ and $\lambda>0$ we have
		\begin{equation*}
			\|f_1-\lambda g_1-(f_2-\lambda g_2)\|\geq \|f_1-f_2\|;
		\end{equation*}
		\item \label{item_CLb} $H$ satisfies the range condition. We say $H\subseteq C_b(E)\times C_b(E\times S)$ satisfies the range condition if for all $\lambda>0$ we have: the uniform closure of $\mathcal{D}(H)$ is a subset of $\mathcal{R}(\bONE-\lambda H)$.
	\end{enumerate}	
	We denote by $R(\lambda)=(\mathbbm{1}-\lambda H)^{-1}$. Then there is a strongly continuous contraction semigroup $V(t)$ defined on the uniform closure of $\mathcal{D}(H)$ and for all $t\geq 0$ and $f$ in the uniform closure of $\mathcal{D}(H)$
	\begin{equation*}
		\lim_{n\to \infty}\left\|R(t/n)^nf-V(t)f\right\|=0.
	\end{equation*}      
\end{lemma}

\begin{condition}[Condition 3.1 in \cite{K2020}]\label{condition:martingale prpblem}
    $A_n\subset C_b(E)\times C_b(E)$ is an operator such that the martingale problem for $A_n\subset C_b(E)\times C_b(E)$ is well-posed. Denote by $\mathbb{P}_x\in\mathcal{P}(D_E(\mathbb{R}^+))$ the solution that satisfies $X(0)=x$, $\mathbb{P}_x$ almost surely. The map $x\mapsto\mathbb{P}_x$ is assumed to be continuous for the weak topology on $\mathcal{P}=\mathcal{P}(D_E(\mathbb{R}^+))$.
\end{condition}

\begin{lemma}[Theorem 3.6 in \cite{K2020}]\label{lem_36}
    Let \Cref{condition:martingale prpblem} be satisfied. For each $h\in C_b(E)$ and $\lambda>0$ the function $R_n(\lambda)h$ is a viscosity solution to $f-\lambda H_nf=h$.
\end{lemma}

	\begin{lemma}[The positive maximum principle for nonlinear generator in \cite{RK2016}]\label{lem_positive_maximum_principle}
	For any $(f_1,g_1)$, $(f_2,g_2)\in H$ and $\lambda>0$, there exist sequences $x_n\in C_b(E)$ satisfied  dissipativity condition
	it is equivalent to (a) and (b).
	\begin{enumerate}
		\item If $x_0\in E$ is such that 
		\begin{equation*}
			f_1(x_0)-f_2(x_0)=\sup_{x_n\in E} f_1(x_n)-f_2(x_n),
		\end{equation*}
		then $g_1(x_0)-g_2(x_0)\leq 0$;
		\item If $x_0\in E$ is such that 
		\begin{equation*}
			f_1(x_0)-f_2(x_0)=\inf_{x_n\in E} f_1(x_n)-f_2(x_n),
		\end{equation*}
		then $g_1(x_0)-g_2(x_0)\geq 0$.
	\end{enumerate}      
\end{lemma}

\begin{lemma}[Theorem 8.27 in \cite{FK2006}]\label{lem_827}
Let $(E,r)$ and $(U,q)$ be complete, separable metric spaces. Suppose that $A\subset C_b(E)\times C(E\times U)$ and $\mathcal{L}:E\times U\to [0,\infty]$ satisfy Conditions 8.9, 8.10, and 8.11 of \cite{FK2006}. Define
\begin{equation*}
    \mathbf{H}f(x)=\sup_{u\in \Gamma_x}(Af(x,\mu)-\mathcal{L}(x,u))
\end{equation*}
with $\mathcal{D}(H)=\mathcal{D}(A)$. 
    Then
    \begin{enumerate}
        \item For each $h\in D_\alpha$, 
        \begin{align*}
            \mathbf{R}_\alpha h(x_0)
          : =\sup_{\gamma\in\mathcal{AC}\atop\gamma(0)=x_0}&\{\int^\infty_0\alpha^{-1}\mathrm{e}^{-\alpha^{-1}t}h(\gamma(t))\mathrm{d}t\\
            &-\int^\infty_0\alpha^{-1}\mathrm{e}^{-\alpha^{-1}t}\int^t_0\mathcal{L}(\gamma(s),\dot{\gamma}(s))\mathrm{d}s\mathrm{d}t
        \end{align*}
        is continuous and is  a solution of $ f-\alpha \mathbf{H}f=h$.
    \end{enumerate}
\end{lemma}

\begin{lemma}[Lemma 8.18 in \cite{FK2006}]\label{lem_818}
    Suppose Conditions 8.9 and 8.10 hold, and let $\{\mathbf{V}(t)\}$ be defined by (8.10). Then for each 
$f\in C_b(E)$ and each $x_0\in E$,
\begin{equation*}
    \mathbf{V}(t)f(x_0)=\lim_{n\to \infty}\mathbf{R}(t/n)^n f(x_0).
\end{equation*}
\end{lemma}

%-----Operator Convergence--------------------------

\section{Operator convergence and principal-eigenvalue problem}\label{se4}
In this section, we discuss how to verify the convergence of the nonlinear operator $H_n$ and study the principal-eigenvalue problem.

\subsection{Operator convergence}

Let $C^2(E\times S;\mathbb{R^+})$ denote the family of all nonnegative functions which are twice differentiable in spatial variable. By \Cref{assu_varn}, we focus on $(X_n(t), \Lambda_n(t))$, which is a Markov process whose generator is given by
\begin{equation}\label{eqn_generator}
	A_nf(x,i)=\eta\left(\mu(i)-x\right)\partial_xf(x,i)+\frac{1}{2n} \theta^2x\partial_{xx}f(x,i)+n\sum_{j\in S }q_{ij}(x)\left(f(x,j)-f(x,i)\right).
\end{equation}
Based on the relation 
\begin{equation*}
H_nf=\frac{1}{n}\mathrm{e}^{-nf}A_n\mathrm{e}^{nf},
\end{equation*}
which is a Fleming's nonlinear generator; see \cite{F1978}, we have 
\begin{align*}
H_nf(x,i)&=\eta\left(\mu(i)-x\right)\partial_{x} f(x,i)+\frac{1}{2}\theta^2x \left(\partial_{x}f(x,i)\right)^2+\frac{1}{2n}\theta^2x\partial_{xx}f(x,i)\\
&\quad+\sum_{j\in S}q_{ij}(x)\big(\mathrm{e}^{n(f(x,j)-f(x,i))}-1\big).
\end{align*}
For any $i\in S$, $H_nf(x,i)$ does not converge as $n\to \infty$. Hence, we can choose a suitable function sequence to deal with the divergent term.
Let 
\begin{equation}\label{eqn_suitable_function}
	f_n(x,i)=f(x)+\frac{1}{n}\phi(x,i),
\end{equation}
 then we have
\begin{align*}
	H_nf_n(x,i)&=\eta\big(\mu(i)-x\big)\big(\partial_{x}f(x)+\frac{1}{n}\partial_{x}\phi(x,i)\big)+\frac{1}{2}\theta^2x\big(\partial_{x}f(x)+\frac{1}{n}\partial_{x}\phi(x,i)\big)^2\\
	&\quad+\frac{1}{2n}\theta^2x\big(\partial_{xx}f(x)+\frac{1}{n}\partial_{xx}\phi(x,i)\big)+\sum_{j\in S}q_{ij}(x)\big(\mathrm{e}^{\phi(x,j)-\phi(x,i)}-1\big).
\end{align*}
Hence, there exists a limiting function $H_{f,\phi}(x,i)$ such that, for all $f\in \mathcal{D}(H)$ and $\phi\in C_b(E\times S)$
\begin{equation*}
    \lim_{n\to \infty}\|H_nf_n-H_{f,\phi}\|=0.
\end{equation*}
where 
\begin{equation}\label{eqn_limit}
    H_{f,\phi}(x,i)=\eta\left(\mu(i)-x\right)\partial_{x} f(x)+\frac{1}{2}\theta^2x \big(\partial_{x}f(x)\big)^2+\sum_{j\in S}q_{ij}(x)\big(\mathrm{e}^{\phi(x,j)-\phi(x,i)}-1\big).
    \end{equation}
% \begin{equation}\label{eqn_limit}
% 	\begin{split}
% 			&H_{f,\phi}(x,i)\\
% 		&=\lim_{n\to\infty}H_nf_n(x,i)\\
% 		&=\eta\left(\mu(i)-x\right)\partial_{x} f(x)+\frac{1}{2}\theta^2x \big(\partial_{x}f(x)\big)^2+\sum_{j\in S}q_{ij}(x)\big(\mathrm{e}^{\phi(x,j)-\phi(x,i)}-1\big).
% 	\end{split}
% \end{equation}
\par
We gather the important results in the following proposition.
\begin{proposition}[Multivalued limit operators]\label{pro_limit operator}
	 Let $(X_n(t),\Lambda_n(t))$ be a Markov process on $E\times S$ with generator $A_n$ from \eqref{eqn_generator}. Then there exists a multivalued limit operator $H$ such that $\|H_n-H\|\to0$ as $n\to \infty$:
 for any $(f,g)\in H$, there exists $(f_n,g_n)\in H_n$ such that $\|f_n-f\|\to 0$ and $\|g_n-g\|\to 0$, and $H$ is given by
		\begin{equation*}
		H:=\Big\{(f,H_{f,\phi})\Big|f\in C_b(E), H_{f,\phi}\in C_b(E\times S)~\text{and}~\phi\in C^2_b(E\times S)\Big\}.
	\end{equation*}
% with 
% \begin{equation*}
% H_{f,\phi}(x,i)	=\eta\left(\mu(i)-x\right)\partial_{x} f(x)+\frac{1}{2}\theta^2x \left(\partial_{x}f(x)\right)^2+\sum_{j\in S}q_{ij}(x)\left(\mathrm{e}^{\phi(x,j)-\phi(x,i)}-1\right).
% \end{equation*}
\end{proposition}
\subsection{Principal-eigenvalue problem}
Based on the preparation of the above subsection, we solve a principal-eigenvalue problem for proving the comparison principle in \Cref{lem_1d}. We first start with the related notation. 
\par
Denote
\begin{equation*}
	B_{x,p}(i):=\eta(\mu(i)-x)p+\frac{1}{2}\theta^2xp^2
\end{equation*}
and define $R_x$ by          
\begin{equation*}
	R_x \mathrm{e}^{\phi(x,i)}:=\sum_{j\in S}q_{ij}(x)(\mathrm{e}^{\phi(x,j)}-\mathrm{e}^{\phi(x,i)}),
\end{equation*}
where $p=\partial_{x}f(x)$.
From \eqref{eqn_limit}, one has
\begin{equation*}
	H_{p,\phi}(x,i)=B_{x,p}(i)+\mathrm{e}^{-\phi(x,i)}R_x\mathrm{e}^{\phi(x,i)}.
\end{equation*}

One approach to dealing with this limit $H$ is to select a function $\phi$ so that the limit is independent of $i$. That is, to find eigenvalue functions $\phi(x,y)$ and functions $g_{f,\phi}(x)$ such that
\begin{equation*}
	g_{f,\phi}(x):=B_{x,p}(i)+\mathrm{e}^{-\phi(x,i)}R_x\mathrm{e}^{\phi(x,i)}.
\end{equation*}
Multiply both sides of the equation by $\mathrm{e}^{\phi(x,i)}$, and fix point $x$. 
\begin{lemma}[Principal-eigenvalue problem]\label{lem_eigen} Let Assumptions \ref{asm_conservative} and \ref{asm_q_conti} be satisfied. 
	For each $(x,p)$, there exist $\bar{\phi}>0$ and a unique eigenvalue $\tau(x,p)\in \mathbb{R}$ such that
\begin{equation}\label{eigenvalue}
	Q_{x,p}\bar{\phi}(i)=\tau(x,p)\bar{\phi}(i),
\end{equation}
with the operator $Q_{x,p}=B_{x,p}+R_x$, where $p=\partial_{x}f(x)$. In detail, $\tau(x,p)$ is given by
\begin{align}\label{eqn_eigenvalue}
	\tau(x,p)=\inf_{\pi\in \mathcal{P}(S)}\sup_{g>0}\int_E\frac{-Q_{x,p}g(z)}{g(z)}\pi(\text{d}z)
	=-\sup_{\pi\in\mathcal{P}(S)}\inf_{g>0}\int_E\frac{Q_{x,p}g(z)}{g(z)}\pi(\text{d}z).
\end{align} 
Moreover, recalling $\mathcal{H}$ in \eqref{eqn_1mathcalH}, since 
\begin{equation}\label{eqn_eigenvalue}
	\mathcal{H}(x,p)=-\tau(x,p),
\end{equation}
as a by-product, one has
\begin{equation}\label{eqn_Hxp}
\begin{split}
	\mathcal{H}(x,p)&=\sup_{\pi\in\mathcal{P}(S)}\inf_{g>0}\int_E\frac{Q_{x,p}g(z)}{g(z)}\pi(\text{d}z)\\
	&=\sup_{\pi\in\mathcal{P}(S)}\big\{\int_E B_{x,p}(z)\pi(\text{d}z)-\mathcal{I}(x,\pi)\big\},
\end{split}
\end{equation}
where
\begin{equation}\label{eqn_I(x,pi)}
	\mathcal{I}(x,\pi)=-\inf_{g>0}\int_E\frac{R_xg(z)}{g(z)}\pi(\text{d}z).
\end{equation}
\end{lemma}

\begin{proof}
Using Assumptions \ref{asm_conservative} and \ref{asm_q_conti}, from the Perron-Frobenius theorem in \cite{DV1975}, we can obtain there exists a unique eigenvalue with associated eigenfunction which have the represent \eqref{eqn_eigenvalue}. 
\end{proof}

%--------Exponential tightness---------------------------

\section{Exponential tightness}\label{se5}
In this section, we prove the exponential tightness. We apply \cite[Corollary 4.17]{FK2006}, which derives exponential tightness from exponential compact containment condition and the convergence of the sequence $ H_n$. We therefore focus on the former and first give its definition.

\begin{definition}[Exponential compact containment condition]
 For each $a>0$ and $T>0$, there exists a compact $K_{a,T}\subset E$ such that 
\begin{equation*}
\limsup_{n\to\infty}\frac{1}{n}{\rm log}\mathbb{P}[X_n(t)\notin K_{a,T}~\mbox{for~some}~t\leq T]\leq-a.
\end{equation*}
\end{definition}

In our paper, we prove that $(X_n(t),\Lambda_n(t))$ satisfies the exponential compact containment condition. 
To do this, we need to find a containment function $\Upsilon$, which plays the role of a Lyapunov function and allows our analysis to be restricted to compact regions in $E$. Here we give the rigorous definition and take a specific function $\Upsilon$ in our case.

\begin{definition}[Containment function]
	We say that a function $\Upsilon:E\to [0,\infty)$ is a containment function for $B_{x,p}$ if $\Upsilon\in C^1(E)$ and it is such that 
	\begin{itemize}
		\item for every $C>0$, the set $\{x~|~\Upsilon(x)\leq C\}$ is compact;
		\item $\sup_{x,i}B_{x,\partial_{x}\Upsilon(x)}(x,i)<\infty$.
	\end{itemize}
\end{definition}
\begin{lemma}
The function \begin{equation}\label{eqn_Upsilon}
	\Upsilon(x):=-\log(x)+\log(1+\frac{1}{2}x^2)	-\log\sqrt{2}
\end{equation} is a containment function for $B_{x,p}$.
\end{lemma}
\begin{proof}
	Firstly, we prove that $\Upsilon$ has compact sub-level sets. Note that $0$ and $\infty$ are the boundary of $E$ and the function $x\mapsto\Upsilon(x)$ goes to $\infty$ at the boundary points $0$ and $\infty$, respectively. Regarding the second property, for any $x\in E$, we have 
\begin{align}\label{eqn_Upsilon_infty}
		H((x,i),\partial_{x}\Upsilon(x))=-\frac{1}{x}\big(\eta\mu(i)+\frac{\theta^2}{2}\big)+\eta<\infty,
	\end{align}
	and which boundedness condition follows with the constant
	\begin{equation*}
		 \sup_{x,i}-\frac{1}{x}\big(\eta\mu(i)+\frac{\theta^2}{2}\big)<\infty.
	\end{equation*}
	From \eqref{eqn_Upsilon_infty}, it follows that 
	\begin{equation}\label{eqn_BBound}
		C_{\Upsilon}:=\sup_{x,i}B_{x,\partial_x\Upsilon(x)}(x,i)<\infty.
	\end{equation}
The proof is completed.
\end{proof}
Here, we are ready to give the proposition that $ (X_n(t),\Lambda_n(t))$ satisfies the exponential compact containment condition.

\begin{proposition}[Exponential compact containment condition]\label{lem_expo_com_con}
	Let  $(X_n(t),\Lambda_n(t))$ be a Markov process corresponding to $A_n$. $\Upsilon$ is a containment function in \eqref{eqn_Upsilon}.
 Suppose that the sequence $(X_n(0),\Lambda_n(0))$ is exponentially tight with speed $n$.
	Then the sequence $(X_n(t),\Lambda_n(t))$ satisfies the exponential compact containment condition with speed $n$: for every $T>0$ and $a \geq 0$, there exists a compact set $K_{a,T}\subset E$ such that
	\begin{equation*}
		\limsup_{n\to\infty}\frac{1}{n}{\rm log}\mathbb{P}\left[(X_n(t),\Lambda_n(t))\notin K_{a,T}\times S ~for~some~t\leq T\right]\leq -a.
	\end{equation*}
\end{proposition}

\begin{proof}
We use the proof method coming from  \cite[Lemma 4.22]{FK2006}. The details are as follows.
\par
Fix $a\geq0$ and $T>0$. $S$ is a finite state space, $S$ is also a compact set. We will construct a compact set $K'\times S$ by Tychonoff's theorem  \cite[Theorem 3.2.4]{R1989MR1039321} such that
	\begin{equation*}
		\limsup_{n\to\infty}\frac{1}{n}{\rm log}\mathbb{P}\left[(X_n(t),\Lambda_n(t)) \notin K'\times S\, for~some~t\leq T\right]\leq -a.
	\end{equation*}
	As $(X_n(0),\Lambda_n(0)$ is exponentially tight with speed $n$, we can find compact $K_o$ so that
	\begin{equation*}
		\limsup_{n\to\infty}\frac{1}{n}{\rm log}\mathbb{P}\left[(X_n(0),\Lambda_n(0))\notin K_0\times S\right]\leq -a,
	\end{equation*}
	%where $z_0=\sqrt{2}$ is a point such that $\Upsilon(z_0)=0$ and $K_0$ is a compact set. 
 Then, in virtue of the convergence of the operator, we can find $(f_n,g_n)\in H_n$, a compact $K\times S$ and an open set $G\times S$, and define
 \begin{equation*}
\beta(K,G,S):=\liminf_{n\to \infty}\left(\inf_{(x,i)\in G^c\times S}f_n(x,i)-\sup_{(x,i)\in K\times S}f_n(x,i)\right)
 \end{equation*}
 and 
 \begin{equation*}
     \gamma(G,S)=\limsup_{n\to \infty}\sup_{x\in G\times S}g_n(x,i)
 \end{equation*}
such that $\beta(K,G,S)+T\gamma(G,S)\leq -a$. 
	
	Set $\gamma:=\sup_{(x,i)\in E\times S}H((x,i),\partial_{x}\Upsilon(x))$ and $c_1:=\sup_{(x,i)\in K_0\times S}\Upsilon(x)$. Observe that $\gamma<\infty$ by \eqref{eqn_Upsilon_infty} and $c_1<\infty$ by compactness. Now choose $c_2$ such that 
	\begin{equation}\label{eqn:construct_compact}
		-(c_2-c_1)+T\gamma=-a
	\end{equation}
	and take $K=\{(x,i)\in E\times S ~|~\Upsilon(x)\leq c_1\}$ and $G=\{(x,i)\in E\times S~|~\Upsilon(x)<c_2\}$.
	
	Let $\theta:[0,\infty)\rightarrow[0,\infty)$ be a compactly supported smooth function with the property that $\theta(x)=x$ for $x\leq c_2$. For each $n$, define $f_n:=\theta\circ\Upsilon$ and $g_n:=H_n f_n$. By the convergence of operator, $g_n\rightarrow Hf$ and moreover, by construction $\beta(K,G,S)=c_2-c_1$ and $\gamma(G,S)=\gamma$. Thus by \eqref{eqn:construct_compact} and \cite[Lemma 4.22]{FK2006}
	 we obtain
	\begin{equation*}
		\limsup_{n\to\infty}\frac{1}{n}{\rm log}\mathbb{P}[(X_n(t),\Lambda_n(t))\notin G\times S~for~some ~t\leq T]\leq -a
	\end{equation*}
	and the compact containment condition holds with $K_{a,T}=\bar{G}$.
\end{proof}
%-----------Comparison principle-----------------------

\section{Comparison principle}\label{se6}
In this section, we prove the comparison principle for the Hamilton-Jacobi equation $f-\lambda Hf=h$ with the help of the single valued operator $\mathbf{H}f(x):=\mathcal{H}(x,\partial_{x}f(x))$ in \eqref{eqn_Hxp} as defined. 
\par
We argue by first encoding the containment function $\Upsilon$ into domain of our operators. This allows us to work with optimizes as in \Cref{rem_compact}, and the strategy is summarized in \Cref{3} below. 
We begin with the definition of the operators $H_1,H_2,H_\dagger$, and $H_\ddagger$ using $	\Upsilon(x)=-\log(x)+\log\left(1+\frac{1}{2}x^2\right)-\log\sqrt{2}$ with $C_{\Upsilon}=\sup_{x,i}B_{x,\partial_x\Upsilon(x)}(x,i)<\infty$.
Denote by $C^{\infty}_l(E)$ the set of smooth functions on $E$ that have a lower bound and by $C^{\infty}_u(E)$ the set of smooth functions on $E$ that have an upper bound.

\begin{definition}(Single valued operators)
\begin{itemize}
\item For $f\in C^{\infty}_l(E)$ and $\delta \in(0,1)$ set
\begin{equation*}
f^{\delta}_{\dagger}:=(1-\delta)f+\delta \Upsilon,
\end{equation*}
\begin{equation*}
H^{\delta}_{\dagger,f}(x):=(1-\delta)\mathbf{H}f(x)+\delta C_{\Upsilon},
\end{equation*}
and set
\begin{equation*}
H_{\dagger}:=\left\{\left(f^{\delta}_\dagger,H^{\delta}_{\dagger,f}\right)\Big |f\in C^\infty_l(E),\delta\in(0,1)\right\}.
\end{equation*}
\item For $f\in C^{\infty}_u(E)$ and $\delta\in(0,1)$ set
\begin{equation*}
f^{\delta}_{\ddagger}:=(1+\delta)f-\delta \Upsilon,
\end{equation*}
\begin{equation*}
H^{\delta}_{\ddagger,f}(x):=(1+\delta)\mathbf{H}f(x)-\delta C_{\Upsilon},
\end{equation*}
and set
\begin{equation*}
H_{\ddagger}:=\Big\{\left(f^\delta_\ddagger,H^\delta_{\ddagger,f}\right)\Big |f\in C^\infty_u(E),\delta\in(0,1)\Big\}.
\end{equation*}
\end{itemize}
\end{definition}
\begin{definition}(Multivalued operators)
\begin{itemize}
\item For $f\in C^{\infty}_l(E)$, $\delta\in(0,1)$ and $\phi\in C^2_b(E\times S)$. Set
\begin{equation*}
f^{\delta}_1:=(1-\delta)f+\delta \Upsilon,
\end{equation*}
\begin{equation*}
H^\delta_{1,f,\phi}(x,i):=(1-\delta)H_{f,\phi}(x,i)+\delta C_{\Upsilon},
\end{equation*}
and set
\begin{equation*}
H_1:=\Big\{\left(f^{\delta}_1,H^\delta_{1,f,\phi}\right)\Big | f\in C^{\infty}_l(E),\delta\in(0,1), \phi\in C^2_b(E\times S)\Big\}.
\end{equation*}
\item For $f\in C^{\infty}_u(E)$, $\delta\in(0,1)$ and $\phi\in C^2_b(E\times S)$. Set
\begin{equation*}
f^{\delta}_2:=(1+\delta)f-\delta \Upsilon,
\end{equation*}
\begin{equation*}
H^\delta_{2,f,\phi}(x,i):=(1+\delta)H_{f,\phi}(x,i)-\delta C_\Upsilon,
\end{equation*}
and set
\begin{equation*}
H_2:=\left\{\left(f^{\delta}_2,H^\delta_{2,f,\phi}\right)\Big | f\in C^{\infty}_u(E),\delta\in(0,1), \phi\in C^2_b(E\times S)\right\}.
\end{equation*}
\end{itemize}
\end{definition}

\begin{figure}[htp]
	\centering
	\begin{tikzpicture}[>=stealth,xscale=0.96,yscale=0.78]
	\node[coordinate] (nw) at (-6.8,2.2) {};
	\node[coordinate] (sw) at (-6.8,-2.2) {};
	\node[coordinate] (se) at (3.5,-2.2) {};
	\node[coordinate] (ne) at (3.5,2.2) {};
	\node[coordinate] (e) at (3.5,0) {};
	\node[coordinate] (w) at (-6.8,0) {};
	%\fill[blue, opacity=0.2] (w) --  (e) {[rounded corners] |- (sw) -- (w)};
	%\fill[red, opacity=0.2] (w) --  (e) {[rounded corners] |- (nw) -- (w)};
	\draw[->, very thick]   (-5,0.2)--node[left=1.5mm,above,sloped]{sub}(-3,1);
	\draw[->, very thick]   (-5,-0.2)--node[left=1.5mm,below,sloped]{super}(-3,-1);
	\draw[->, very thick]   (-2.5,1)--node[above=1mm]{sub}(0,1);
	\draw[->, very thick]   (-2.5,-1)--node[below=1mm]{super}(0,-1);
	\draw[<-, very thick]   (0.55,-1)--node[right=1.5mm,below,sloped]{super}(2.5,-0.2);
	\draw[<-, very thick]   (0.55,1)--node[right=1.5mm,above,sloped]{sub}(2.5,0.2);
	\fill[gray, draw=black,rounded corners,fill opacity=0.3] (-0.7,-1.35) rectangle (1.3,1.5);
	\node at (-5.5,0){$\Large \text{ H}$};
 \node at (-6.5,0.28){$ \text{implicit}$};
 \node at (-6.6,-0.3){$ \text{multivalued}$};
 \fill[blue, draw=black,rounded corners,fill opacity=0.3] (-7.6,0.8) rectangle (-5,-0.8);
 \node at (-2.8,1){$\Large \text{ H}_1$};
	\node at (-2.8,-1){$\Large \text{ H}_2$};
	\node at (0.2,-1){$\Large \text{ H}_\ddagger$};
	\node at (0.2,1){$\Large \text{ H}_\dagger$};
	\node at (2.7,0){$\Large \textbf{ H}$};
 \node at (4,0.28){$\text{explicit}$};
  \node at (4.2,-0.3){$ \text{single~valued}$};
  \fill[blue, draw=black,rounded corners,fill opacity=0.3] (2.5,0.8) rectangle (5.3,-0.8);
	%\node at (-5.9,1.7){$\Large \text{ sub}$};
	%\node at (-5.9,-1.7){$\Large \text{ super}$};
	\node at (0.3,0){comparison };
%	\draw[rounded corners] (nw) rectangle (se);
	\end{tikzpicture}
	\caption{An arrow connecting an operator $A$ with operator $B$ with subscript `sub’ means that viscosity subsolutions of $f-\lambda Af=h$ are also viscosity subsolutions of $f-\lambda Bf=h$. Similarly, we get the description for arrows with a subscript `super'. The middle gray box around the operators $H_\dagger$ and $H_\ddagger$ indicates that the comparison principle holds for subsolutions of $f-\lambda H_\dagger f=h$ and supersolutions of $f-\lambda H_\ddagger f=h$. The left blue box indicates $H$ is an implicit and multivalued operator. The right blue box indicates $\mathbf{H}$ is an explicit single value operator.}
	\label{3}
\end{figure}
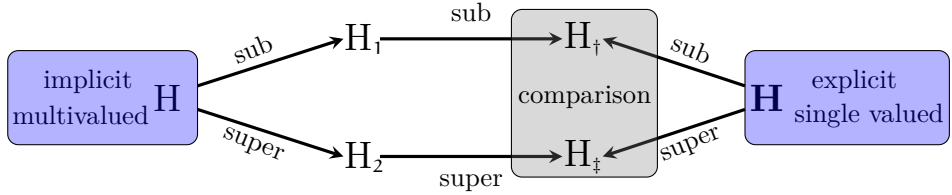

Based on the above preparations, we are ready to state the most important proposition of this section.
\begin{proposition}[Comparison principle]\label{pro_compa_prin}
Let \Cref{asm_q_conti} be satisfied. Let $h_1$, $h_2\in C_b(E)$ and $\lambda>0$. Let $u$ be any subsolution to $f-\lambda H f=h_1$ and let $v$ be any supersolution to $f-\lambda Hf=h_2$. Then we have that 
	\begin{equation*}
		\sup_{x}u(x)-v(x)\leq\sup_{x}h_1(x)-h_2(x).
	\end{equation*}
\end{proposition}

\subsection{Strategy of proof of \Cref{pro_compa_prin}}
The argument of \Cref{pro_compa_prin} is inspired by the methods of \cite[Chapter 11]{FK2006} and \cite[Section 5]{KS2021} and is carried out by establishing the \Cref{3}. We first establish the two horizontal arrows in \Cref{3}.

%In the proof, we further understand that $\textbf{H}$ is explicit single valued operator and $H$ is implicit multivalued valued operator. The existence of viscosity solution of $f-\lambda\textbf{H}f=h$ and the comparison principle of the subsolution of $f-\lambda H_{\dagger}f=h_1$ and the supersolution of $f-\lambda H_{\ddagger}f=h_2$ imply the existence and uniqueness of viscosity solution of $f-\lambda Hf=h$. On the one hand, the viscosity subsolution of $f-\lambda Hf=h$ is the viscosity subsolution of $f-\lambda H_1f=h_1$, see \Cref{lem_h1}; the viscosity subsolution of $f-\lambda H_1f=h_1$ is the viscosity subsolution of  $f-\lambda H_{\dagger}f=h_1$, see \Cref{lem_1d}; the viscosity subsolution of $f-\textbf{H}f=h_1$ is the viscosity subsolution of  $f-\lambda H_{\dagger}f=h_1$, see \Cref{lem_bfd}. On the other hand, it is similar for viscosity supersolution.

\begin{lemma}\label{lem_1d}
Fix $\lambda>0$ and $h\in C_b(E)$.
\begin{enumerate}
\item Every subsolution to $f-\lambda H_1f=h$ is also a subsolution to $f-\lambda H_{\dagger}f=h$.
\item Every supersolution to $f-\lambda H_1f=h$ is also a supersolution to $f-\lambda H_{\ddagger}f=h$.
\end{enumerate}
\end{lemma}

\begin{proof}
    Let $u$ be a subsolution to $f-\lambda H_1f=h$. We show it is also a subsolution to $f-\lambda H_{\dagger}f=h$.
    To do it, for one thing, we find a unique optimizer $x_0$ in the compact level sets of $\Upsilon$ for the definition of viscosity solutions due to the existence of $\Upsilon$. For another, we find a corrector using $x_0$.
    \par
   Step 1: we show there exists $x_0$ such that \begin{equation}\label{eq11.1}
	u(x_0)-f^{\delta}_1(x_0)=\sup_{x}u(x)-f^{\delta}_1(x).
\end{equation}
 First of all, note that $u$ and $-f^\delta_1$ are upper semicontinuous. As $\Upsilon$ has compact sub-level sets, there exists $x_0$ such that
 \begin{equation*}
   	u(x_0)-f^{\delta}_1(x_0)=\sup_{x}u(x)-f^{\delta}_1(x).  
 \end{equation*}
 \par
 Next, let $\hat{f}\in C_l^\infty(E)$ such that $\hat{f}(x_0)=f(x_0)$ and $\hat{f}(x)>f(x)$ if $x\neq x_0$ so that $x_0$ is the unique optimizer in
  \begin{equation*}
   	u(x_0)-f^{\delta}_1(x_0)=\sup_{x}u(x)-f^{\delta}_1(x)
 \end{equation*}
 and in addition $f'(x_0)=\hat{f}'(x_0)$.
 \par
 Step 2: we consider the corrector. The corrector $\phi_{x_0}=\tau(x_0,p)$ existing by \Cref{lem_eigen} is such that
 \begin{equation*}
     H^\delta_{1,\hat{f},\phi_{x_0}}(x_0,i)
 \end{equation*}
 does not depend on $i$, and we have
 \begin{equation*}
    H^\delta_{1,\hat{f},\phi_{x_0}}(x_0,i)=(1-\delta)\mathbf{H}f(x_0)+\delta C_\Upsilon.
 \end{equation*}
 due to \Cref{lem_eigen}.
   As $u$ is a subsolution to $f-\lambda H_1 f=h$, there are $(x_n,i_n)$ such that 
    \begin{equation*}
        \lim_{n}u(x_n)-\hat{f}^\dagger_\delta(x_n)=\sup_x u(x)-\hat{f}^\dagger_\delta(x)
    \end{equation*}
    and
    \begin{equation*}
        \limsup_n u(x_n)-\lambda H^\delta_{1,\hat{f},\phi_{x_0}}(x_n,i_n)-h(x_n)\leq 0.
    \end{equation*}
    
    As $x_0$ is the unique optimizer of $\sup(u-\hat{f}^\dagger_\delta)$, and as $S$ is compact, there exists $i_0\in S$ such that along a subsequence we have $(x_n,i_n)\to(x_0,i_0)$ as $n\to \infty$. We conclude that
    \begin{equation*}
      u(x_0)-\lambda H^\delta_{\dagger,f}(x_0)-h(x_0)=u(x_0)-\lambda H^\delta_{1,\hat{f},\phi_{x_0}}(x_0,i)-h(x_0)\leq 0.
    \end{equation*}
    So that in combination with \eqref{eq11.1}, we have obtained the two desired properties for each pair of functions in $H_\dagger$. We conclude that $u$ is a subsolution to $f-\lambda H_\dagger f=h$.
\end{proof}

\begin{lemma}\label{lem_bfd}
Fix $\lambda>0$ and $h\in C_b(E)$. 
\begin{enumerate}
\item Every subsolution to $f-\lambda \mathbf{H}f=h$ is also a subsolution to $f-\lambda H_\dagger f=h$.
\item Every supersolution to $f-\lambda \mathbf{H}f=h$ is also a supersolution to $f-\lambda H_{\ddagger} f=h$.	
\end{enumerate}
\end{lemma}
\begin{proof}
Fix $\lambda>0$ and $h\in C_b(E)$. Let $u$ be a subsolution to $f-\lambda \mathbf{H}f=h$. We prove it is also a subsolution to $f-\lambda H_\dagger f=h$.
			
Fix $\delta>0$, $f\in C^{\infty}_{\ell}(E)$ such that $(f^{\delta}_\dagger,H^{\delta}_{\dagger,f})\in H_\dagger$. We will prove that there is a sequence $x_n\in E$ such that
\begin{equation}\label{eq3.1}
\lim_{n\to \infty}u(x_n)-f^{\delta}_{\dagger}(x_n)=\sup_{x\in E}u(x)-f^{\delta}_{\dagger}(x),
\end{equation}
\begin{equation}\label{eq6.6}
\limsup_{n\to\infty}u(x_n)-\lambda H^{\delta}_{\dagger,f}(x_n)-h(x_n)\leq 0.
\end{equation}
As the function $[u-(1-\delta)f]$ is bounded from above and $\Upsilon$ has compact sublevel sets, the sequence $x_n$ along which the first limit is attained can be assumed to lie in the compact set
\begin{equation*}
K:=\big\{x~|~\Upsilon(x)\leq \delta^{-1}\sup_{x}(u(x)-(1-\delta)f(x))\big\}.
\end{equation*}
Set $M=\delta^{-1}\sup_{x}(u(x)-(1-\delta)f(x))$. Let $\gamma: \mathbb{R}\to\mathbb{R}$ be a smooth increasing function such that
\begin{equation*}
\gamma(r)=
\begin{cases}
r ,&if ~\mbox{$r\leq M$,}\\
M+1,&if~\mbox{$r\geq M+2$.}
\end{cases}
\end{equation*}
		
Let $f_\delta$ be a function on $E$ defined by 
\begin{equation*}
f_{\delta}(x):=\gamma((1-\delta)f(x)+\delta\Upsilon(x))=\gamma(f^{\delta}_\dagger(x)).
\end{equation*}
			
By construction, $f_\delta$ is smooth and constant outside of a compact set and thus lies in $\mathcal{D}(H)=C^\infty_{cc}(E)$. As $u$ is a viscosity subsolution for $f-\lambda Hf=h$ there exists a sequence $x_n\in K\subseteq E$(by our choice of $K$) with
\begin{equation}\label{eq3.3}
\lim_{n}u(x_n)-f_{\delta}(x_n)=\sup_{x\in E}u(x)-f_{\delta}(x),
\end{equation}
\begin{equation}\label{eq3.4}
	\limsup_{n}u(x_n)-\lambda\mathbf{H}f_{\delta}(x_n)-h(x_n)\leq 0.
\end{equation}
As $f_\delta$ equals $f^{\delta}_\dagger$ on $K$, we have from \eqref{eq3.3} that also
\begin{equation*}
	\lim_{n}u(x_n)-f^{\delta}_\dagger(x_n)=\sup_{x\in E}u(x)-f^{\delta}_\dagger(x),
\end{equation*}
establishing \eqref{eq3.1}. Convexity of $p\longmapsto\mathcal{H}(x,p)$ yields for arbitrary points $x\in K$ the estimate
\begin{align*}
	\mathbf{H}f_\delta(x)&=\mathcal{H}(x,\partial_{x}f_\delta(x))\\
	&\leq(1-\delta)\mathcal{H}(x,\partial_{x} f(x))+\delta\mathcal{H}(x,\partial_{x}\Upsilon(x))\\
	&\leq (1-\delta)\mathcal{H}(x,\partial_{x} f(x))+\delta C_\Upsilon=H^\delta_{\dagger,f}(x).
\end{align*}
Combining this inequality with \eqref{eq3.4} yields
\begin{align*}
	\limsup_{n}& u(x_n)-\lambda H^{\delta}_{\dagger,f}(x_n)-h(x_n)\\
	&	\leq \limsup_{n}u(x_n)-\lambda \mathbf{H}f_\delta(x_n)-h(x_n)\leq 0,
\end{align*}
establishing \eqref{eq6.6}. This concludes the proof.
\end{proof}			

\begin{lemma}\label{lem_h1}
Fix $\lambda >0$ and $h\in C_b(E)$.
\begin{enumerate}
\item Every subsolution to $f-\lambda Hf=h$ is also a subsolution to $f-\lambda H_1f=h$.
\item  Every supersolution to $f-\lambda Hf=h$ is also a supersolution to $f-\lambda H_2f=h$.
\end{enumerate}
\end{lemma}

\begin{proof}
This proof has the same idea as \Cref{lem_bfd}, but we need to make appropriate modifications. To maintain integrity and readability, we give its proof in the following.
\par
Fix $\lambda >0$ and $h\in C_b(E)$. Let $u$ be a subsolution to $f-\lambda Hf=h$. We prove it is also a subsolution to $f-\lambda H_1f=h$. Fix $\delta\in(0,1)$, $\phi\in C^2_b(E\times S)$ and $f\in C_l^\infty(E)$, such that $(f^\delta_1,H^\delta_{1,f,\phi})\in H_1$. We will prove that there are $(x_n,i_n)$ such that
\begin{equation}\label{eq6.1}
\lim_{n}u(x_n)-f^{\delta}_1(x_n)=\sup_{x}u(x)-f^{\delta}_1(x)
\end{equation}
\begin{equation}\label{eq6.2}
\limsup_{n}u(x_n)-\lambda H^\delta_{1,f,\phi}(x_n,i_n)-h(x_n)\leq 0.
\end{equation}
We have that $M:=\delta^{-1}\sup_{x}(u(x)-(1-\delta)f(x))<\infty$ as $u$ is bounded and $f\in C_l(E)$. It follows that the sequence $x_n$ along which the limit in \eqref{eq6.1} is attained is contained in the compact set $K:=\{x~|~\Upsilon(x)\leq M\}$.
\par	
Let $\gamma:\mathbb{R}\rightarrow\mathbb{R}$ be a smooth increasing function such that
\begin{equation*}
\gamma(r)=
\begin{cases}
r ,&if ~\mbox{$r\leq M$,}\\
M+1,&if~\mbox{$r\geq M+2$.}
\end{cases}
\end{equation*}
Denote by $f_\delta$ the function on $E$ defined by
\begin{equation*}
f_\delta(x):=\gamma((1-\delta)f(x)+\delta\Upsilon(x))=\gamma(f^{\delta}_1(x)).
\end{equation*}
By construction, $f_\delta$ is smooth and constant outside of a compact set and thus lies in $\mathcal{D}(H)=C^\infty_{cc}(E)$. As $\mathrm{e}^\phi\in C^2_b(E\times S)$, we also have $\mathrm{e}^{(1-\delta)\phi}\in C^2(E\times S)$. 
We conclude that $(f_\delta,H_{f_\delta,(1-\delta)\phi})\in H$. 
As $u$ is a viscosity subsolution for $f-\lambda Hf=h$, there exist $x_n\in K\subseteq E$(by our choice of K) and $i_n\in S$ with
\begin{equation}\label{6.2}
\lim_nu(x_n)-f_\delta(x_n)=\sup_xu(x)-f_\delta(x),
\end{equation}
\begin{equation}\label{eq6.3}
\limsup_n u(x_n)-\lambda H_{f_\delta,(1-\delta)\phi}(x_n,i_n)-h(x_n)\leq 0.
\end{equation}
As $f_\delta$ equals $f^{\delta}_1$ on $K$, we have from \eqref{6.2} that also
\begin{equation*}
\lim_n u(x_n)-f^{\delta}_1(x_n)=\sup_x u(x)-f^{\delta}_1(x),
\end{equation*}
establishing \eqref{eq6.1}.  For arbitrary sequences $(x_n,i_n)$ the elementary estimate 
\begin{align*}
&H_{f_\delta,(1-\delta)\phi}(x_n,i_n)\\
&~=B_{x,\partial_{x}f_\delta(x)}(x_n,i_n)+\mathrm{e}^{-(1-\delta)\phi(x_n,i_n)}R_x\mathrm{e}^{(1-\delta)\phi(x_n,i_n)}\\
&~\leq (1-\delta)B_{x,\partial_{x}f(x)}(x_n,i_n)+\delta B_{x,\partial_{x}\Upsilon(x)}(x_n,i_n)+(1-\delta)\mathrm{e}^{-\phi(x_n,i_n)}R_x\mathrm{e}^{\phi(x_n,i_n)}\\
&~=(1-\delta)\left(B_{x,\partial_{x}f(x)}(x_n,i_n)+\mathrm{e}^{-\phi(x_n,i_n)}R_x\mathrm{e}^{\phi(x_n,i_n)} \right)+\delta B_{x,\partial_{x}\Upsilon(x)}(x_n,i_n)\\
&~\leq H^\delta_{1,f,\phi}(x_n,i_n),
\end{align*}
In the first inequality, we use that $ B_{x,p}$ is convex concerning $p$. In virtue of \eqref{eqn_BBound}, the last inequality holds.
Combining above inequality with \eqref{eq6.3} yields
\begin{align*}
\limsup_n&\left [u(x_n)-\lambda H^\delta_{1,f,\phi}(x_n,i_n)-h(x_n)\right]\\
&\leq\limsup_n \left[u(x_n)-\lambda H_{f_\delta,\phi}(x_n,i_n)-h(x_n)\right]\leq 0,
\end{align*}
establishing \eqref{eq6.2}. This concludes the proof.
\end{proof}		

The following lemma is to verify the comparison principle for Hamilton-Jacobi-Bellman equations involving the Hamiltonian $H_\dagger$ and $H_\ddagger$. 
\begin{lemma}\label{Hd}
Assumptions \ref{asm_conservative} and \ref{asm_q_conti} hold. Let $h_1$, $h_2\in C_b(E)$ and $\lambda>0$. Let $u$ be any subsolution to $f-\lambda H_{\dagger}f=h_1$ and let $v$ be any supersolution to $f-\lambda H_{\ddagger}f=h_2$. Then we have
\begin{equation*}
\sup_{x}u(x)-v(x)\leq\sup_{x}h_1(x)-h_2(x).
\end{equation*}
\end{lemma}

Key step in the proof is doubling of variables procedure as e.g. explained in \cite{CIL1992}. We first give the the definition of penalization function for introducing the auxiliary lemma below that often used in proving \Cref{Hd}.

\begin{definition}[Penalization function]
We say that $d:E\times E\to [0,\infty)$ is a penalization function if $d \in C(E\times E)$ and if $x=y$ if and only if $d(x,y)=0$.
\end{definition}
\begin{lemma}[Lemma A.10 in \cite{CK_2017}]\label{lem_pro3.7}
Let $u$ be bounded and upper semicontinuous, let $v$ be bounded and lower semicontinuous, let the distance function $d:E^2\to \mathbb{R}^+$ be good penalization function and let $\Upsilon$ be a good containment function.
\par
Fix $\delta>0$. For every $m>0$ there exist points $x_{\delta,m}$, $y_{\delta,m}\in E$, such that
\begin{align*}
&\frac{u(x_{\delta,m})}{1-\delta}-\frac{v(y_{\delta,m})}{1+\delta}-md^2(x_{\delta,m},y_{\delta,m})-\frac{\delta}{1-\delta}\Upsilon(x_{\delta,m})-\frac{\delta}{1+\delta}\Upsilon(y_{\delta,m})\\
&=\sup_{x,y\in E}\big\{\frac{u(x)}{1-\delta}-\frac{v(y)}{1+\delta}-md^2(x,y)-\frac{\delta}{1-\delta}\Upsilon(x)-\frac{\delta}{1+\delta}\Upsilon(y)\big\}.
\end{align*}
Additionally, for every $\delta>0$ we have that
\begin{enumerate}
\item The set $\{x_{\delta,m},y_{\delta,m} ~|~m >0\}$ is relatively compact in $E$.
\item All limit points of $\{(x_{\delta,m},y_{\delta,m})\}_{m >0}$ are of the form $(z,z)$ and for these limit points we have 
\begin{equation*}
 u(z)-v(z)=\sup_{x\in E}u(x)-v(x).   
\end{equation*}
\item We have
\begin{equation*}
\lim_{m\to\infty}md^2(x_{\delta,m},y_{\delta,m})=0.
\end{equation*}
\end{enumerate}
\end{lemma}
\begin{remark}
  For the good penalization function $d$ in \Cref{lem_pro3.7}, we take  
\begin{equation}\label{eqn_distance}
    d^2(x,y)=\frac{2}{\theta^2}(\sqrt{x}-\sqrt{y})^2
\end{equation}
based on \cite[Section 2.2]{DFL2011} in the proof of \Cref{Hd}. 
\end{remark}
Here, we give the proof of \Cref{Hd}.
\begin{proof}[Proof of \Cref{Hd}] 
For $0<\delta<1$ and $m>1$, let
\begin{equation*}
\Phi_{\delta,m}(x,y):=\frac{u(x)}{1-\delta}-\frac{v(y)}{1+\delta}-md^2(x,y)-\frac{\delta}{1-\delta}\Upsilon(x)-\frac{\delta}{1+\delta}\Upsilon(y),
\end{equation*}		
where $d(\cdot,\cdot)$ is given in \eqref{eqn_distance} and $\Upsilon(\cdot)$ is given in \eqref{eqn_Upsilon}.  
 Since $\Upsilon(\cdot)$ has compact level sets, there exists $(x_{\delta,m},y_{\delta,m})\in E\times E$ satisfying 
\begin{equation}\label{eqn_Phi_sup}
\Phi_{\delta,m}(x_{\delta,m},y_{\delta,m})=\sup_{(x,y)\in E\times E}\Phi_{\delta,m}(x,y).
\end{equation}
Let $\varphi_1^{\delta,m}\in\mathcal{D}(H_\dagger)$ be defined as
\begin{align*}%\label{eqn_phi}
\varphi_1^{\delta,m}(x):&=(1-\delta)\big(\frac{v(y_{\delta,m})}{1+\delta}+md^2(x,y_{\delta,m})+\frac{\delta}{1-\delta}\Upsilon(x)+\frac{\delta}{1+\delta}\Upsilon(y_{\delta,m})\big)\\
&\quad+(1-\delta)(x-x_{\delta,m})^2,
\end{align*}
where adding $(1-\delta)(x-x_{\delta,m})^2$ in $\varphi_1^{\delta,m}(x)$ implies that $u-\varphi_1^{\delta,m}$ attains its a unique supremum at $x =x_{\delta,m}$, namely
\begin{align*}
\sup_{x\in E}u(x)-\varphi_1^{\delta,m}(x)=u(x_{\delta,m})-\varphi_1^{\delta,m}(x_{\delta,m}).
\end{align*}
By the viscosity subsolution property of $u$ one has
\begin{equation}\label{eqn_sub_inequality}
u(x_{\delta,m})-\lambda\left[(1-\delta)\mathcal{H}(x_{\delta,m},p^1_{\delta,m})+\delta C_{\Upsilon}\right]\leq h_1(x_{\delta,m}),
\end{equation}
where
\begin{equation}\label{eqn_P1} P^1_{\delta,m}:=m\partial_{x} d^2(x_{\delta,m},y_{\delta,m})=\frac{2m}{\theta^2}\big(1-\frac{\sqrt{y_{\delta,m}}}{\sqrt{x_{\delta,m}}}\big).
\end{equation}
Similarly, let $\varphi_2^{\delta,m}\in\mathcal{D}(H_\ddagger)$ be defined as
\begin{align*}
\varphi^{\delta,m}_2(y):&=(1+\delta)\big(\frac{u(x_{\delta,m})}{1-\delta}-md^2(x,y_{\delta,m})-\frac{\delta}{1-\delta}\Upsilon(x_{\delta,m})-\frac{\delta}{1+\delta}\Upsilon(y)\big)\\
&\quad-(1+\delta)(y-y_{\delta,m})^2.
\end{align*}
Therefore, we obtain the supersolution inequality
\begin{equation}\label{eqn_super_inequality}
v(x_{\delta,m})-\lambda\left[(1+\delta)\mathcal{H}(y_{\delta,m},p^2 _{\delta,m})-\delta C_{\Upsilon}\right]\geq h_2(y_{\delta,m}),
\end{equation}
where
\begin{equation}\label{eqn_P2}
p^2_{\delta,m}:=-m\partial_y d^2(x_{\delta,m},y_{\delta,m})=-\frac{2m}{\theta^2}\big(1-\frac{\sqrt{x_{\delta,m}}}{\sqrt{y_{\delta,m}}}\big).
\end{equation}
By \Cref{lem_pro3.7}, we have
\begin{equation}\label{eq_md}
\lim_{m\to\infty}md^2(x_{\delta,m},y_{\delta,m})=0.
\end{equation}
Combining \eqref{eqn_sub_inequality}, \eqref{eqn_super_inequality} and \eqref{eq_md} we get
\begin{align}
\sup_{x\in E}u(x)-v(x)&\leq \liminf_{\delta\to0}\liminf_{m\to\infty}\big(\frac{u(x_{\delta,m})}{1-\delta}-\frac{v(y_{\delta,m})}{1+\delta}\big)\notag\\
&\leq \liminf_{\delta\to0}\liminf_{m\to\infty}\big\{\frac{h_1(x_{\delta,m})}{1-\delta}-\frac{h_2(y_{\delta,m})}{1+\delta}\label{eqn_k1}\\
&\quad+\frac{\delta}{1-\delta}C_{\Upsilon}+\frac{\delta}{1+\delta}C_{\Upsilon}\label{eqn_k2}\\
&\quad+ \lambda\big(\mathcal{H}(x_{\delta,m},p^1_{\delta,m})-\mathcal{H}(y_{\delta,m},p^2_{\delta,m})\big)\big\},
\end{align}
where in the first inequality we use \eqref{eqn_Phi_sup} and  drop the non-negative functions  $d^2(\cdot,\cdot)$ and $\Upsilon(\cdot)$.
The term in \eqref{eqn_k2} vanishes as $C_{\Upsilon}$ in \eqref{eqn_BBound} is a constant.
\par
Based on \Cref{lem_pro3.7}, for fixed $\delta$ and varying $m$, the sequence $(x_{\delta,m},y_{\delta,m})$ takes its values in a compact set and, hence, admits converging subsequences. By (b) of \Cref{lem_pro3.7}, these subsequences converge to points of the form $(x,x)$. Therefore, we can deal with \eqref{eqn_k1}.
Then, by the above analysis, we can get 
\begin{align*}
	\sup_{x\in E}u(x)-v(x)&\leq \lambda\liminf_{\delta\to0}\liminf_{m\to\infty} \big(\mathcal{H}(x_{\delta,m},p^1_{\delta,m})-\mathcal{H}(y_{\delta,m},p^2_{\delta,m})\big)\\
	&\quad+\sup_{x\in E}h_1(x)-h_2(x).
\end{align*}
It follows that the comparison principle holds for $f-\lambda H_\dagger f=h_1$ and $f-\lambda H_\ddagger f=h_2$ whenever for any $\delta>0$
\begin{equation}\label{eqn_H-H}
 	\liminf_{m\to\infty} \big(\mathcal{H}(x_{\delta,m},p^1_{\delta,m})-\mathcal{H}(y_{\delta,m},p^2_{\delta,m})\big)\leq0.
\end{equation}
\par
To that end, recall  $\mathcal{H}(x,p)$ in \eqref{eqn_Hxp}:
\begin{equation*}
\mathcal{H}(x,p)=\sup_{\pi\in\mathcal{P}(S)}\big\{\int_E B_{x,p}(z)\pi(\text{d}z)-\mathcal{I}(x,\pi)\big\},
\end{equation*}
where $\pi\mapsto \int_EB_{x_{\delta,m},p^i_{\delta,m}}(z)\pi(\mathrm{d}z)$ is bounded and continuous and $\mathcal{I}(x_{\delta,m},\cdot)$ has compact sub-level sets in $\mathcal{P}(S)$. Thus, there exists an optimizer $\pi_{\delta,m}\in\mathcal{P}(S)$ such that 
\begin{equation}\label{eqn_H1}
	\mathcal{H}(x_{\delta,m},p^1_{\delta,m})=\int B_{x_{\delta,m},p^i_{\delta,m}}(z)\pi_{\delta,m}(\text{d}z)-\mathcal{I}(x_{\delta,m},\pi_{\delta,m})
\end{equation}
and 
\begin{equation}\label{eqn_H2}
	\mathcal{H}(y_{\delta,m},p^2_{\delta,m})\leq \int_E B_{x_{\delta,m},p^i_{\delta,m}}(z)\pi_{\delta,m}(\text{d}z)-\mathcal{I}(x_{\delta,m},\pi_{\delta,m}).
\end{equation}
Combining \eqref{eqn_H1} and \eqref{eqn_H2}, we obtain 
\begin{align}
	\mathcal{H}&(x_{\delta,m},p^1_{\delta,m})-\mathcal{H}(y_{\delta,m},p^2_{\delta,m})\notag\\
	&\leq \int_E\big(B_{x_{\delta,m},p^1_{\delta,m}}(z)-B_{y_{\delta,m},p^2_{\delta,m}}(z)\big)\pi_{\delta,m}(\text{d}z)\label{eqn_B_B}\\
	&\quad+\mathcal{I}(y_{\delta,m},\pi_{\delta,m})-\mathcal{I}(x_{\delta,m},\pi_{\delta,m})\label{eqn_I-I}.
\end{align}
It is sufficient to prove that \eqref{eqn_B_B} and \eqref{eqn_I-I} are sufficiently small. For \eqref{eqn_B_B}, by calculating the difference of integrand  $B_{x,p}$ in detail, for any $z \in E$ and $i\in S$, one has
\begin{align*}
&B_{x_{\delta,m},p^1_{\delta,m}}(z)-	B_{y_{\delta,m},p^2_{\delta,m}}(z)
\\&=\big(\eta(\mu(z)-x_{\delta,m})p^1_{\delta,m}+\frac{1}{2}\theta^2x_{\delta,m}(p^1_{\delta,m})^2\big)-\big(\eta(\mu(z)-y_{\delta,m})p^2_{\delta,m}+\frac{1}{2}\theta^2y_{\delta,m}(p^2_{\delta,m})^2\big)\\&=\big(\eta(\mu(z)-x_{\delta,m})p^1_{\delta,m}-\eta(\mu(z)-y_{\delta,m})p^2_{\delta,m}\big)+\big(\frac{1}{2}\theta^2x_{\delta,m}(p^1_{\delta,m})^2- \frac{1}{2}\theta^2y_{\delta,m}(p^2_{\delta,m})^2 \big)\\
&=\frac{2m\eta}{\theta^2}(\mu(z)-y_{\delta,m})\big(1-\frac{\sqrt{x_{\delta,m}}}{\sqrt{y_{\delta,m}}}\big)+\frac{2m\eta}{\theta^2}(\mu(z)-x_{\delta,m})\big(1-\frac{\sqrt{y_{\delta,m}}}{\sqrt{x_{\delta,m}}}\big)\\
&=\frac{2m\eta}{\theta^2}\big[(\mu(z)-y_{\delta,m})\big(1-\frac{\sqrt{x_{\delta,m}}}{\sqrt{y_{\delta,m}}}\big)+(\mu(z)-x_{\delta,m})\big(1-\frac{\sqrt{y_{\delta,m}}}{\sqrt{x_{\delta,m}}}\big) \big]\\
&=-m\eta\mu(z)\frac{d^2(x_{\delta,m},y_{\delta,m})}{\sqrt{x_{\delta,m}y_{\delta,m}}}-m\eta d^2(x_{\delta,m},y_{\delta,m})
\leq 0,
\end{align*}
where in the third equality we use \eqref{eqn_P1} and \eqref{eqn_P2}.
For \eqref{eqn_I-I}, we utilize the equi-continuity of $\mathcal{I}(\cdot,\pi)$ established in \Cref{lem_Lip} below for the spatial variable. We are left with stating and verifying \Cref{lem_Lip}. This finishes the proof of \eqref{eqn_H-H} and the comparison principle for $H_\dagger$ and $H_\ddagger$.
\end{proof}

\begin{lemma}\label{lem_Lip}
Let Assumption \ref{asm_q_conti} be satisfied. Recall \eqref{eqn_I(x,pi)}:
\begin{align*}
\mathcal{I}(x,\pi)&=-\inf_{g>0}\int_{E}\frac{R_{x}g(z)}{g(z)}\pi(\mathrm{d}z).
\end{align*}
For any compact set $G\subseteq E$, then the collection $\{x\mapsto J(x,\pi)\mid \pi\in \mathcal{P}(S)\}$ is equi-continuity.
\end{lemma}
\begin{proof}
	Let $\rho$ be some metric on the topology of $E$. We will prove  that for any compact sets $G\subseteq E$ and $\varepsilon>0$, there is some $\delta>0$ such that for all $x,y\in G$ with $\rho(x,y)\leq \delta$ and for all $\pi\in \mathcal{P}(S)$, we have
	\begin{equation}\label{eqn_I_conti}
		|\mathcal{I}(x,\pi)-\mathcal{I}(y,\pi)|\leq \varepsilon.
	\end{equation}
	In virtue of the definition of $\mathcal{I}$, there exists a function $\phi\in C(S)$ independent of $x$ such that $\mathrm{e}^\phi $ in the domain of $R_x$, and
	\begin{align*}
		\mathcal{I}(x,\pi)&=-\inf_{g>0}\int_{E}\frac{R_{x}g(z)}{g(z)}\pi(\mathrm{d}z)\\
		&=\sup_{\phi\in C(S)}\sum_{i,j\in S}q_{ij}(x)\pi_i\big(1-\mathrm{e}^{\phi(j)-{\phi(i)}}\big).
	\end{align*}
Let $x,y\in G$.
 By continuity of the transition rate $q_{ij}(x)$, the $\mathcal{I}(x,\cdot)$ are uniformly bounded for $x\in G$:
\begin{equation*}
	0\leq \mathcal{I}(x,\pi)\leq  \sum_{i,j,i\neq j}q_{ij}(x)\pi_i\leq  \sum_{i,j,i\neq j}q_{ij}(x)\leq  \sum_{i,j,i\neq j}\bar{q}_{ij},~\bar{q}_{ij}:=\sup_{x\in G}q_{ij}(x).
\end{equation*} 
For any $n\in N$, there exists $\phi^n\in C(S)$ such that
\begin{equation*}
0\leq \mathcal{I}(x,\pi)\leq  \sum_{i,j,i\neq j}q_{ij}(x)\pi_i(1-\mathrm{e}^{\phi^n(j)-\phi^n(i)})+\frac{1}{n}.
\end{equation*}
By reorganizing, we find for all pairs $(k,l)$ the bound
\begin{equation*}
\pi_k\mathrm{e}^{\phi^n(l)-\phi^n(k)}\leq \frac{1}{r_{G}(k,l)}\big(\sum_{i,j,i\neq j}q_{ij}(x)\pi_i+\frac{1}{n}\big)\leq \frac{1}{r_{G}(k,l)}\big(\sum_{i,j,i\neq j}\bar{q}_{ij}+\frac{1}{n}\big),
\end{equation*}
where $r_{G}(k,l):=\inf_{x\in G}q_{kl}(x)$. Thereby, evaluating in $\mathcal{I}(y,\pi)$ the same function $\phi^n$ to estimate the supremum,
\begin{align*}
\mathcal{I}&(x,\pi)-\mathcal{I}(y,\pi)\\
&\leq\frac{1}{n}+\sum_{k,l,k\neq l}q_{kl}(x)\pi_k\big(1-\mathrm{e}^{\phi^n(l)}-\mathrm{e}^{\phi^n(k)}\big)-\sum_{k,l,k\neq l}q_{kl}(y)\pi_k\big(1-\mathrm{e}^{\phi^n(l)}-\mathrm{e}^{\phi^n(k)}\big)\\
&\leq\frac{1}{n}+\sum_{k,l,k\neq l}|q_{kl}(x)-q_{kl}(y)|\pi_k+\sum_{k,l,k\neq l}\left|q_{kl}(y)-q_{kl}(x)\right|\pi_k \mathrm{e}^{\phi^n(l)-\phi^n(k)}\\
&\leq \frac{1}{n}+\sum_{k,l,k\neq l}|q_{kl}(x)-q_{kl}(y)|\big[1+\frac{1}{r_{G}(k,l)}\big(\sum_{k,l,k\neq l}\bar{q}_{kl}+1\big)\big].
\end{align*}
We take $n\to \infty$ and use that the rates $x\mapsto q_{kl}(x)$ are Lipschitz continuous under Assumption \ref{asm_q_conti}, and hence uniformly continuous on compact sets, to obtain \eqref{eqn_I_conti}.
Hence, the proof of the lemma is concluded.
\end{proof}

\subsection{Proof of \Cref{pro_compa_prin}}
We now prove \Cref{pro_compa_prin}; that is, the verification of the comparison principle for the Hamilton-Jacobi equations $f-\lambda Hf=h$. The proof follows the strategy of \Cref{3} combined with \Cref{lem_1d}, \Cref{lem_bfd}, \Cref{lem_h1} and \Cref{Hd}. We thus obtain \Cref{4} via adding these lemmas in \Cref{3} as below for an easy understanding of the proof strategy of  \Cref{pro_compa_prin}. 
\begin{figure}[htp]
	\centering 
\begin{tikzpicture}[>=stealth,xscale=0.96,yscale=0.78]
	\node[coordinate] (nw) at (-6.8,2.2) {};
	\node[coordinate] (sw) at (-6.8,-2.2) {};
	\node[coordinate] (se) at (3.5,-2.2) {};
	\node[coordinate] (ne) at (3.5,2.2) {};
	\node[coordinate] (e) at (3.5,0) {};
	\node[coordinate] (w) at (-6.8,0) {};
	\fill[blue, opacity=0.2] (w) --  (e) {[rounded corners] |- (sw) -- (w)};
	\fill[red, opacity=0.2] (w) --  (e) {[rounded corners] |- (nw) -- (w)};
	\draw[->, very thick]   (-5,0.2)--node[left=1.5mm,above,sloped]{$\text{Lem~\ref{lem_h1}}$}(-3,1);
	\draw[->, very thick]   (-5,-0.2)--node[left=1.5mm,below,sloped]{$\text{Lem~\ref{lem_h1}}$}(-3,-1);
	\draw[->, very thick]   (-2.5,1)--node[above=1.5mm]{$\text{Lem~\ref{lem_1d}}$}(0,1);
	\draw[->, very thick]   (-2.5,-1)--node[below=1.5mm]{$\text{Lem~\ref{lem_1d}}$}(0,-1);
	\draw[<-, very thick]   (0.55,-1)--node[right=1.5mm,below,sloped]{$\text{Lem~\ref{lem_bfd}}$}(2.5,-0.2);
	\draw[<-, very thick]   (0.55,1)--node[right=1.5mm,above,sloped]{$\text{Lem~\ref{lem_bfd}}$}(2.5,0.2);
	\fill[gray, draw=black,rounded corners,fill opacity=0.3] (-0.6,-1.3) rectangle (1.1,1.3);
	\node at (-5.5,0){$\Large \text{ H}$};
	\node at (-2.8,1){$\Large \text{ H}_1$};
	\node at (-2.8,-1){$\Large \text{ H}_2$};
	\node at (0.2,-1){$\Large \text{ H}_\ddagger$};
	\node at (0.2,1){$\Large \text{ H}_\dagger$};
	\node at (2.7,0){$\Large \textbf{ H}$};
	\node at (-5.9,1.7){$\Large \text{ sub}$};
	\node at (-5.9,-1.7){$\Large \text{ super}$};
	\node at (0.3,0){$\text{Lem~\ref{Hd}}$};
	\draw[rounded corners] (nw) rectangle (se);
	\end{tikzpicture}
	\caption{Add lemmas in Figure \ref{3}}
	\label{4}
\end{figure}
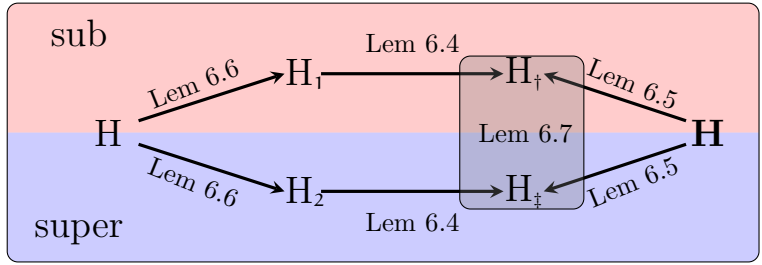
\begin{proof}[Proof of \Cref{pro_compa_prin}]  
Fix $h_1$, $h_2\in C_b(E)$ and $\lambda >0$. Let $u$ be a viscosity subsolution to $(1-\lambda H)f=h_1$ and $v$ be a viscosity supersolution to $(1-\lambda H)f=h_2$. By  \Cref{lem_h1} and \Cref{lem_1d}, the function $u$ is a viscosity subsolution to $(1-\lambda H_\dagger)f=h_1$ (see red part on \Cref{4}) and $v$ is a viscosity supersolution to $(1-\lambda H_\ddagger)f=h_2$ (see blue part on \Cref{4}). Hence by the comparison principle for $H_\dagger$, $H_\ddagger$ established in \Cref{Hd}, we get $\sup_x u(x)-v(x)\leq \sup_{x}h_1(x)-h_2(x)$. This finished the proof.
\end{proof}
%-------------Proof of action-integral representation of ---------------------------------------
\section{Proof of action-integral representation of the rate function}\label{se7}
 In this section, we will prove our main result \Cref{thm_LDP}. According to the strategy of proof of \Cref{thm_LDP} in \Cref{se44}, the proof is based on three main parts that we have proven:
 \begin{itemize}
     \item Operator convergence;
    \item Exponential tightness;
    \item Comparison principle.
 \end{itemize}
Hence, $X^\varepsilon_n(t)$ satisfies large deviation principles with projective limit form rate function \eqref{eqn_rate_function_1}.
\par
We are left to prove that \eqref{eqn_rate_function_1} has the action-integral form rate function \eqref{eqn_AC}, and put it in the proof of \Cref{thm_LDP} below. To achieve the aim, according to step 4 in \Cref{sub_outline}, we should first prove the lemma below which is necessary to obtain \eqref{eqn_AC}.

\begin{lemma}\label{lem_89810811}
    Let $\mathcal{H}:E\times E\to \mathbb{R}$ be the map given in \eqref{eqn_Hxp} and the operator $\mathbf{H}f(x):=\mathcal{H}(x,\partial_x f(x))$. Then, the operator $\mathbf{H}$ satisfies Condition 8.9, 8.10, and 8.11 of \cite{FK2006}.
       % \item For all $\lambda>0$ and $h\in C(\mathbb{R}^d)$, the comparison principle holds for $(1-\lambda \mathbf{H})u=h$.
    \end{lemma}
\begin{proof}
 We first show that the following conditions imply Conditions 8.9, 8.10, and 8.11 in a non-compact setting. These ideas come from the proof of Proposition 6.1-(i) in \cite{PS2021}. We begin with modifying the conditions adapted to our setting.
  \begin{enumerate}[(1)]
    \item The function $\mathcal{L}: E\times E \to  [0,\infty]$ is lower semicontinuous and for every $C\geq 0$, the level
set $\{ (x, v)\in E\times E~ |~\mathcal{L}(x, v)\leq C \}$ is relatively compact in $E\times E $.
\item For all $f \in \mathcal{D}(\mathbf{H})$ there exists a right continuous, nondecreasing function $\psi_f:[0, \infty)\to [0, \infty)$ such that for all $(x, v) \in E\times E$,
\begin{equation*}
    |\partial_x f(x)\cdot v|\leq \psi_f(\mathcal{L}(x,v))~~~and ~~~\lim_{r\to \infty} r^{-1}\psi_f(r)=0.
\end{equation*}
\item For each $x_0\in E$ and every $f \in  \mathcal{D}(\mathbf{H})$, there exists an absolutely continuous path $x: [0,\infty) \mapsto E$ such that 
\begin{equation*}
 \int^t_0 \mathcal{H}(x(s), \partial_x f(x(s))) \mathrm{d}s = \int^t_0 \partial_xf(x(s)) \cdot \dot{x}(s)-\mathcal{L}(x(s),\dot{x}(s))] \mathrm{d}s .  
\end{equation*}
\end{enumerate}
Then, we will use (1), (2), and (3) to prove Condition 8.9, 8.10, and 8.11.
Regarding Condition 8.9, the operator $Af(x,v):=\partial_x f(x)\cdot v$ on the domain $\mathcal{D}(A) = \mathcal{D}(H)$ satisﬁes Condition 8.9.1 in \cite{FK2006}. For Condition 8.9.2 in \cite{FK2006}, we can choose $\Gamma = E\times E$, and for $x_0 \in  E$, take the pair $(x,\lambda)$ with $x(t) = x_0$ and $\lambda (\mathrm{d}v \times  \mathrm{d}t) = \delta_0 (\mathrm{d}v)\times \mathrm{d}t$. Condition 8.9.3 in \cite{FK2006} is a consequence of Condition 8.9.1 in \cite{FK2006} from above. Condition 8.9.4 in \cite{FK2006} can be verified as follows. Let be $\Upsilon$ the containment function used in \eqref{eqn_Upsilon} and note that the sub-level sets of $\Upsilon$ are compact. Let $\gamma \in \mathcal{AC}(E)$ with $\gamma(0)\in K$ and such that the control
\begin{equation*}
    \int^T_0\mathcal{L}(\gamma (s)), \dot{\gamma}(s)) \leq M
\end{equation*}
implies $\gamma(t)\in \hat{K}$ for all $t\leq T$, where $\hat{K}$ is a compact set. Then,
\begin{align*}
\Upsilon(\gamma(t))&=\Upsilon(\gamma(0))+\int^t_0\langle \partial_x\Upsilon(\gamma (s)), \dot{\gamma}(s)\rangle \mathrm{d}s\\
    &\leq \Upsilon(\gamma(0))+\int^t_0 \mathcal{L}(\gamma (s)), \dot{\gamma}(s) )+\mathcal{H}(\gamma (s), \partial_x \Upsilon(\gamma(s)))\mathrm{d}s\\
    &\leq \sup_{y\in K}\Upsilon(y)+M+\int^T_0\sup_z\mathcal{H}(z,\partial_x \Upsilon(z))\mathrm{d}s\\
    &:=C<\infty.
\end{align*}
Hence, we can take $\hat{K} = \{z \in E | \Upsilon(z)\leq C \} $. Condition 8.9.5 in \cite{FK2006}  is implied by Condition 8.9.2 in \cite{FK2006} from above. 
\par
Condition 8.10 \cite{FK2006} is implied by Condition 8.11 \cite{FK2006} with the fact that $\mathbf{H}1 = 0$ \cite[Remark 8.12-(e)]{FK2006}. 
\par
Finally, Condition 8.11 in \cite{FK2006} is implied by (3) above, with the control $\lambda(\mathrm{d}v \times \mathrm{d}t) = \delta_{\dot{x}(t)} (\mathrm{d}v)\times \mathrm{d}t$.
\end{proof}
In the rest of this section, we prove \Cref{thm_LDP}.
\begin{proof}[Proof of \Cref{thm_LDP}] From \Cref{proposition_main}, we have proven LDP with a projective limit form rate function. Then we rewrite the rate-function on the Skorohod space in a action-integral form. 
\par
We first show that the Lagrangian $\mathcal{L}$ is superlinear, which means $(\mathcal{L}(x,v)/ | v | ) \to \infty$ as $| v | \to \infty $. To do it, for any $c > 0$ we have
\begin{align*}
    \frac{\mathcal{L}(x,v)}{|v|}&=\sup_{p\in \mathbb{R}}\left[p\cdot \frac{v}{|v|}-\frac{\mathcal{H}(x,p)}{|v|}\right]\\
    &\geq\sup_{|p|=c}\left[p\cdot \frac{v}{|v|}-\frac{\mathcal{H}(x,p)}{|v|}\right]\\
    &\geq c-\frac{1}{|v|}\sup_{|p|=c}\mathcal{H}(x,p).
\end{align*}
The convex Hamiltonian is continuous, and therefore $\sup_{|p|=c}\mathcal{H}(x,p)$ is finite. Hence for arbitrary $c > 0$, we have $\mathcal{L}(x, v)/ | v | > c/2 $ for all $| v |$ large enough.

   Let $x:[0,T]\to E$ be absolutely continuous and are two arbitrary $t_1 < t_2$. We show  that
   \begin{equation}\label{eqn_sum_integral}
       I^V_{t_1-t_0}(x(t_1)|x(t_0))=\inf_{\gamma(t_0)=x(t_0)\atop \gamma(t_1)=x(t_1)}\int^t_0 \mathcal{L}(\gamma(s),\dot{\gamma}(s))\mathrm{d}s,
   \end{equation}
   where the infimum is taken over absolutely continuous paths $\gamma: [t_0, t_1 ] \to E$. Once we have this equality established, we obtain for arbitrary $k\in \mathbb{N}$ and points in time $t_0,t_1,\ldots, t_k = T$ the estimate
\begin{equation*}
    I^V_{t_1-t_0}(x(t_1)|x(t_0))+I^V_{t_2-t_1}(x(t_2)|x(t_1))+\cdots+I^V_{t_k-t_{k-1}}(x(t_k)~|~x(t_{k-1})\leq \int^T_0\mathcal{L}(x(s),\dot{x}(s))\mathrm{d}s,
\end{equation*}
   since $x(\cdot)$ satisfies the begin- and endpoint constraints. For the reverse inequality, we note that adding time points increases the two-point rate functions since we add a condition on the paths; for $t_0 < t_1 < t_2 $,
\begin{align*}
   {I}^V_{t_2-t_0}(x(t_2)|x(t_0))&=\inf_{\gamma(t_0)=x(t_0)\atop\gamma(t_2)=x(t_2)}\left[\int^{t_1}_{t_0}\mathcal{L}(\gamma(s),\dot{\gamma}(s))\mathrm{d}t+\int^{t_2}_{t_1}\mathcal{L}(\gamma(s),\dot{\gamma}(s))\mathrm{d}t\right]\\
    &\leq \inf_{\gamma(t_0)=x(t_0)\atop\gamma(t_1)=x(t_1)}\left[\int^{t_1}_{t_0}\mathcal{L}(\gamma(s),\dot{\gamma}(s))\mathrm{d}t\right]+\inf_{\gamma(t_1)=x(t_1)\atop\gamma(t_2)=x(t_2))}\left[\int^{t_2}_{t_1}\mathcal{L}(\gamma(s),\dot{\gamma}(s))\mathrm{d}t\right]\\
    &={I}^V_{t_2-t_1}(x(t_2)|x(t_1))+{I}^V_{t_1-t_0}(x(t_1)|x(t_0)).
\end{align*}
The partitions of a time interval $[0, T ]$ give rise to a monotonically increasing sequence. In the limit, we obtain
\begin{equation*}
\sup_k\sup_{t_i}\sum^k_{i=0}{I}^V_{t_{i+1}-t_{i}}(x(t_{i+1})|x(t_{i}))=\int^T_0\mathcal{L}(x(s),\dot{x}(s))\mathrm{d}s.
\end{equation*}
We do not show that here, but refer to \cite[Definition 7.11, Example 7.12]{V2008}. We now show how \eqref{eqn_sum_integral} follows from the compact sub-level sets. 

For $f\in C_b(E)$ and $x(t_0)\in E$, starting from
	\begin{equation}\label{eqn_VVV}
	\begin{split}
		V(t_1)f(x(t_0))&=\mathbf{V}(t_1)f(x(t_0))\\
		&=\sup_{\gamma(t_0)=x(t_0)\atop\gamma(t_1)= x(t_1)}\big\{f(\gamma(t_1))-\int^{t_1}_{t_0}\mathcal{L}(\gamma(s),\dot{\gamma}(s))\mathrm{d}s\big\}.\\
	&=-\inf_{\gamma(t_0)=x(t_0)\atop\gamma(t_1)= x(t_1)}\big\{\int^{t_1}_{t_0}\mathcal{L}(\gamma(s),\dot{\gamma}(s))\mathrm{d}s-f(\gamma(t_1))\big\}.
		\end{split}
		\end{equation}
  We have
\begin{equation}\label{eqn_condition_rate}
  \begin{split}
			&I^V_{t_1-t_0}(x(t_1)|x(t_0))\\
			&=\sup_{f\in C_b(E)}(f(x(t_1))-\mathbf{V}(t_1)f(x(t_0)))\\
&\overset{\eqref{eqn_VVV}}=\sup_{f\in C_b(E)}\inf_{\gamma(t_0)=x(t_0)\atop\gamma(t_1)= x(t_1)}\big[f(x(t_1))-f(\gamma(t_1))+\int^{t_1}_{t_0}\mathcal{L}(\gamma(s),\dot{\gamma}(s)\mathrm{d}s\big].
		\end{split}
  \end{equation}
  
  For any $f\in C_b(E)$,
  \begin{equation*}
      \inf_{\gamma(t_0)=x(t_0)}\left[f(x(t_1))-f(\gamma(t_1))+\int^{t_1}_{t_0}\mathcal{L}(\gamma(s),\dot{\gamma}(s)\mathrm{d}s\right]\leq \inf_{\gamma(t_0)=x(t_0)\atop \gamma(t_1)=x(t_1)}\int^{t_1}_{t_0}\mathcal{L}(\gamma(s),\dot{\gamma}(s)\mathrm{d}s,
  \end{equation*}
  since $\{\gamma:\gamma(t_0)=x(t_0)\}$ contains $\{\gamma :\gamma(t_0)=x(t_0),\gamma(t_1)=x(t_1)\}$. Taking the supremum over all $f$ 
    \begin{equation*}
      \sup_{f\in C_b(E)}\inf_{\gamma(t_0)=x(t_0)}\left[f(x(t_1))-f(\gamma(t_1))+\int^{t_1}_{t_0}\mathcal{L}(\gamma(s),\dot{\gamma}(s)\mathrm{d}s\right]\leq \inf_{\gamma(t_0)=x(t_0)\atop \gamma(t_1)=x(t_1)}\int^{t_1}_{t_0}\mathcal{L}(\gamma(s),\dot{\gamma}(s)\mathrm{d}s,
  \end{equation*}shows the inequality ``$\leq$'' of \eqref{eqn_sum_integral}.

  For the reverse, let $f\in C_b(E)$. There are curves $\gamma_m$ satisfying $\gamma_m(t_0)=x(t_0)$ and 
  \begin{multline*}
       \inf_{\gamma(t_0)=x(t_0)}\left[f(x(t_1))-f(\gamma(t_1))+\int^{t_1}_{t_0}\mathcal{L}(\gamma(s),\dot{\gamma}(s)\mathrm{d}s\right]+\frac{1}{m}\\
       \geq f(x(t_1))-f(\gamma_m(t_1))+\int^{t_1}_{t_0}\mathcal{L}(\gamma_m(s),\dot{\gamma}_m(s)\mathrm{d}s.
  \end{multline*}
  Since $f$ is bounded, this implies $\limsup_{m\to \infty}\int^{t_1}_{t_0}\mathcal{L}(\gamma_m(s),\dot{\gamma}_m(s)\mathrm{d}s<\infty$. By compactness of sublevel sets, we can pass to a converging subsequence (denoted as well by $\gamma_m$). If $\gamma_m(t_1)\nrightarrow x(t_1)$, then $I_t(x(t_1)~|~x(t_0))=\infty$, and the desired estimate holds. If $\gamma_m(t_1)\to x(t_1)$, then by lower semicontinuity of $\gamma\longmapsto \int^{t_1}_{t_0}\mathcal{L}(\gamma(s),\dot{\gamma}(s)\mathrm{d}s$,
  \begin{align*}
   \inf_{\gamma(t_0)=x(t_0)}&\left[f(x(t_1))-f(\gamma(t_1))+\int^{t_1}_{t_0}\mathcal{L}(\gamma(s),\dot{\gamma}(s)\mathrm{d}s\right]\\
   &\geq \liminf_{m\to \infty} f(x(t_1))-f(\gamma_m(t_1))+\int^{t_1}_{t_0}\mathcal{L}(\gamma_m(s),\dot{\gamma}_m(s)\mathrm{d}s\\
   &\geq \int^{t_1}_{t_0}\mathcal{L}(\gamma(s),\dot{\gamma}(s)\mathrm{d}s\geq \inf_{\gamma(t_0)=x(t_0)\atop \gamma(t_1)=x(t_1)}\int^{t_1}_{t_0}\mathcal{L}(\gamma(s),\dot{\gamma}(s)\mathrm{d}s,
   \end{align*}
   and the reverse inequality ``$\geq$'' of \eqref{eqn_sum_integral} follows.
\end{proof}

%--------Proof of  existence and uniqueness---------------------------

\section{Proof of existence and uniqueness}\label{se9}
In order to prove \Cref{pro_exist_uniq}, we use the methods of Skorokhod's representation and pathwise splicing. In the following, we start with introducing Skorokhod's representation of fast process $\Lambda^\varepsilon_n(t)$.
\subsection{Skorokhod's representation}
The role of Skorokhod's representation is to represent the evolution of the discrete component $\Lambda^{\varepsilon}_n(t)$ in the form of a stochastic integral with respect to a Poisson random measure (see, for example, \cite{S2018}). 
Precisely, for each $x\in \mathbb{R}$, construct a family of intervals $\{ \Upgamma^{\varepsilon}_{ij}(x): i,j\in S \}$  on the half line in the following manner:
\begin{align*}
	\Upgamma^{\varepsilon}_{12}(x)&=\big[0,~\frac{1}{\varepsilon}q_{12}(x)\big)\\
	\Upgamma^{\varepsilon}_{13}(x)&=\big[\frac{1}{\varepsilon}q_{12}(x),~ \frac{1}{\varepsilon}[q_{12}(x)+q_{13}(x)]\big)\\
	&~~\vdots\\
	\Upgamma^{\varepsilon}_{1N}(x)&=\big[\frac{1}{\varepsilon}\sum_{j=1}^{N-1}q_{1j}(x),~\frac{1}{\varepsilon}q_1(x)\big)\\
	\Upgamma^{\varepsilon}_{21}(x)&=\big[\frac{1}{\varepsilon}q_1(x),~ \frac{1}{\varepsilon}[q_1(x)+q_{21}(x)]\big)\\
	\Upgamma^{\varepsilon}_{23}(x)&=\big[\frac{1}{\varepsilon}[q_1(x)+q_{21}(x)],~\frac{1}{\varepsilon}[q_1(x)+q_{21}(x)+q_{23}(x)]\big)\\
	&~~\vdots
\end{align*}
and so on. Therefore, we obtain a sequence of consecutive, left-closed, right-open intervals $\Upgamma^{\varepsilon}_{ij}(x)$ of $\mathbb{R}^+$, each having length $\frac{1}{\varepsilon}q_{ij}(x)$. For convenience of notation, we set $\Upgamma^{\varepsilon}_{ii}(x)= \emptyset$ and $\Upgamma^{\varepsilon}_{ij}(x)=\emptyset$ if $q_{ij}(x)=0$.  
Define a function $h^{\varepsilon}:\mathbb{R}\times S \times M^{\varepsilon}\rightarrow\mathbb{R}$ by
\begin{equation*}
	h^{n,\varepsilon}(x,i,z)=\sum_{l\in S}(l-i)\bm{1}_{{\Upgamma}^{\varepsilon}_{il}(x)}(z).
\end{equation*} 
That is, with the partition $\{\Gamma^{\varepsilon}_{ij}(x): i,j\in S~ \mbox{with}~i\neq j\}$ used and for each $i\in S$, if $x\in \Gamma^{\varepsilon}_{ij}(z),h^{\varepsilon}(x,i,z)=j-i$; otherwise $h^{\varepsilon}(x, i, z) = 0$. 
Then \eqref{eqn_fasting_switch} is equivalent to
\begin{equation}\label{eqn_integral_form}
	\text{d}\Lambda^{\varepsilon}_n (t) = \int_{[0,M^{\varepsilon}]}h^{\varepsilon}(X^\varepsilon_n(t),\Lambda^{\varepsilon}(t-),z)N(\text{d} t,\text{d} z),
\end{equation}
where $M^{\varepsilon} = N (N - 1)H^{\varepsilon}$ with $H^{\varepsilon}:=\max_{i,j\in S}\sup_{x\in \mathbb{R}}\frac{1}{\varepsilon}q_{ij}(x)<\infty$. $N(\text{d} t,\text{d} z)$ is a Poisson random measure(corresponding to a stationary point process $p(t)$) with intensity $\text{d} t \times \bm{m}(\text{d} z)$, and $\bm{m}(\text{d}z)$ is the Lebesgue measure on $[0,M^{\varepsilon}]$. Note that $N(\cdot,\cdot)$ does not depend on $\varepsilon$, because the function $h^{\varepsilon}(\cdot,\cdot,\cdot)$ contains all information about $\frac{1}{\varepsilon}q(\cdot)$, see \cite[Proporsition 2.4]{XZ2017}. $N(\cdot,\cdot)$ is independent of the Brownian motion $W(\cdot)$. 
Due to the finiteness of $\bm{m}(\cdot)$ on $[0, M^{\varepsilon}]$, there is only a finite number of jumps of the process $p(t)$ in each finite time interval.
Let $\sigma_1<\sigma_2<\ldots<\ldots$ be the enumeration of all elements in the domain $D_{p}$ of the stationary point process $p(t)$ corresponding to the above Poisson random measure $N(\text{d} t,\text{d} z)$. It follows that $\lim_{n\to\infty}\sigma_n=\infty$ almost surely.
\begin{remark}
	\eqref{eqn_fasting_switch} describes the evolution of the jump process, but it is difficult to study the existence and uniqueness of the system solution directly by using \eqref{eqn_fasting_switch}. Skorokhod's representation is a good approach to express the phenomenon \eqref{eqn_fasting_switch} with the integral equation, and the information contained is not lost.
\end{remark}
\subsection{Proof of \Cref{pro_exist_uniq}}
For each $k\in S$, when we fixed a state,  \eqref{eqn_CIR} becomes a  classical CIR process
\begin{equation}\label{eqfixk}
	\text{d} X^{\varepsilon,(k)}_n(t)=\eta(\mu(k)-X^{\varepsilon,(k)}_n(t))\text{d} t+n^{-\frac{1}{2}}\theta\sqrt{X^{\varepsilon,(k)}_n(t)}\text{d} W(t).
\end{equation} 
There exist a unique nonnegative strong solution  of \eqref{eqfixk} with $2\eta u(i)\geq \theta^2$ due to  
\cite[Proposition 5.2.13 and Corollary 5.3.23]{KS1991} by Yamada-Watanabe theorem.
\begin{proof}[Proof of  \Cref{pro_exist_uniq}] 
	The idea of proof gets from the stationarity of Brownian motion and Poisson point process, and we segment the path-space through stopping time, see \Cref{2}.
	\begin{figure}[htp]
		\centering
		\begin{tikzpicture}[>=stealth,xscale=0.8,yscale=0.5]
			\fill[red, draw=black,rounded corners,fill opacity=0.3] (-5,-0.2) rectangle (-3,0.2);
			\node at (-5,-0.5){$\Large 0$};
			\fill[blue, draw=black,rounded corners,fill opacity=0.3] (-3,-0.2) rectangle (0,0.2);
			\node at (-3,-0.5){$\Large \sigma_1$};
			\fill[brown, draw=black,rounded corners,fill opacity=0.3] (0,-0.2) rectangle (1.5,0.2);
			\node at (-3,-0.5){$\Large \sigma_1$};
			\node at (0,-0.5){$\Large \sigma_2$};
			\node at (1.5,-0.5){$\Large \sigma_3$};
			\node at (2.5,0){$\Large \ldots$};
			\fill[green, draw=black,rounded corners,fill opacity=0.3] (3.5,-0.2) rectangle (5.5,0.2);
			\node at (3,-0.5){$\Large \sigma_{n-1}$};
			\node at (5.5,-0.5){$\Large \sigma_{n}$};
			\node at (6.5,0){$\Large \ldots$};
		\end{tikzpicture}
		\caption{The time interval}
		\label{2}
	\end{figure}
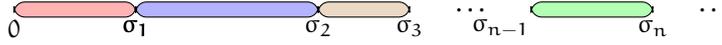
\begin{description}
    \item[Step 1:] 
	Let us first consider the solution in the time interval $[0,\sigma_1]$, where $\sigma_1$ is the stopping time. For any $t\in [0,\sigma_1)$ and any path $\{ (X^{\varepsilon}_n(s),\Lambda^{\varepsilon}_n(s)):0\leq s\leq t\}$, we always have
	\begin{equation}
		\int^t_0\int_{[0,M^{\varepsilon}]}(X^{\varepsilon}_n(s-),\Lambda^{\varepsilon}_n(s-),z)N_1(\text{d} s,\text{d} z)\equiv0, ~~~\mbox{and}~\mbox{then}~~\Lambda^{\varepsilon}_n(t)\equiv\Lambda^{\varepsilon}_n(0)=k.
	\end{equation}
	Hence, on the interval $[0,\sigma_1)$, \eqref{eqn_CIR} is equivalent to \eqref{eqfixk} which has a unique strong solution $X^{\varepsilon,(k)}_n(t)$ with $X^{\varepsilon,(k)}_n(0)=x$, so $(X^{\varepsilon}_n(t),\Lambda^{\varepsilon}_n(t))=(X^{\varepsilon,(k)}_n(t),k)$ for $0\leq t<\sigma_1$. From \eqref{eqn_integral_form} we have
	\begin{equation*}
		\Lambda^{\varepsilon}_n(\sigma_1)=k+\sum_{l\in S}(l-k){\bf 1}_{\Gamma^{\varepsilon}_{kl}(X^{\varepsilon,(k)}_n(\sigma_1))}(p(\sigma_1)).
	\end{equation*}
	Then, on the time interval $[0,\sigma_1]$, set	
	\begin{equation}\label{sigma1}
		( X^{\varepsilon}_n(t),\Lambda^{\varepsilon}_n(t))=
		\begin{cases}
			(X^{\varepsilon,(k)}_n(t),k), &\mbox{$0\leq t<\sigma_1$,}\\
			(X^{\varepsilon,(k)}_n(\sigma_1),\Lambda^{\varepsilon}_n(\sigma_1)), &\mbox{$t=\sigma_1$}
		\end{cases}
	\end{equation}
	Next, set $\tilde{\xi}=X^{\varepsilon}_n(\sigma_1)$, $\tilde{W}(t)=W(t+\sigma)-W(t)$ and $\tilde{p}(t)=p(t+\sigma_1)$. 
	\item[Step 2:]
	Similarly, we consider the solution on the interval $[0,\sigma_2-\sigma_1]$ with respect to $(\tilde{\xi},\Lambda^{\varepsilon}_n(\sigma_1))$ as above, and define
	\begin{equation*}
		(\tilde{X}^{\varepsilon}_n(t),\tilde{\Lambda}^{\varepsilon}_n(t))=(X^{\varepsilon,(\Lambda^{\varepsilon}_n(\sigma_1))}_n(t),\Lambda^{\varepsilon}_n(\sigma_1))~~~~\mbox{for}~~~ 0\leq t<\sigma_2-\sigma_1,
	\end{equation*}
	\begin{equation*}
		\tilde{X}^{\varepsilon}_n(\sigma_2-\sigma_1)=X^{\varepsilon,(\Lambda^{\varepsilon}(\sigma_1))}_n(\sigma_2-\sigma_1),
	\end{equation*}
	\begin{equation*}
		\tilde{\Lambda}^{\varepsilon}_n(\sigma_2-\sigma_1)=\Lambda^{\varepsilon}_n(\sigma_1)+\sum_{l\in S}(l-\Lambda^{\varepsilon}_n(\sigma_1)){\bf 1}_{\tilde{A}^{\varepsilon}_n(l)}(\tilde{p}(\sigma_2-\sigma_1)),
	\end{equation*}
	where
	\begin{equation*}
		\tilde{A}^{\varepsilon}_n(l)=\Gamma^{\varepsilon}_{\Lambda^{\varepsilon}_n(\sigma_1)l}(X^{\varepsilon,(\Lambda^\varepsilon_n(\sigma_1))}_n(\sigma_2-\sigma_1)-).
	\end{equation*}
	Furthermore, we define
	\begin{equation*}
		(X^{\varepsilon}_n(t),\Lambda^{\varepsilon}_n(t))=(\tilde{X}^{\varepsilon}_n(t-\sigma_1),\tilde{\Lambda}^{\varepsilon}_n(t-\sigma_1))~~~t\in[\sigma_1,\sigma_2],
	\end{equation*}
	which and \eqref{sigma1} together give the solution on the time interval $[0,\sigma_2]$. Continuing this procedure inductively, $(X^{\varepsilon}_n(t),\Lambda^{\varepsilon}_n(t))$ is determined uniquely on the time interval $[0,\sigma_n]$ for every $n$ and thus $(X^{\varepsilon}_n(t),\Lambda^{\varepsilon}_n(t))$ is determined globally because $\lim_{n\to\infty}\sigma_n=\infty$ almost surely.
	\item[Step 3:] Consequently, we have proved the existence of a unique strong solution to the systems \eqref{eqn_CIR} and \eqref{eqn_integral_form}. 
\end{description}
The proof is completed.
\end{proof}

%----------------Appendix-----------------

%-----------Acknowledgments-----------------
\section*{Acknowledgments}
%\noindent{\bf Acknowledgments:} 
The research of Y. Hu was supported by the China Scholarship Council (CSC) and  the BIT Research and Innovation Promoting Project (Grant No. 2022YCXZ037).
The research of  F. Xi was supported by the National Natural Science Foundation of China (Grant No. 12071031). 

%%%%%%%%%%%%%%%%%%%%%%%%%%%%%%%%%%%%%%%%%%%%%%%%%%\cite{mao1998robust}
%	\bibliography{2022hr1-reference}
%	\bibliographystyle{acm}
%	plain，按字母的顺序排列，比较次序为作者、年度和标题
%	
%	2. unsrt，样式同plain，只是按照引用的先后排序
%	
%	3. alpha，用作者名首字母+年份后两位作标号，以字母顺序排序
%	
%	4. abbrv，类似plain，将月份全拼改为缩写，更显紧凑
%	
%	5. ieeetr，国际电气电子工程师协会期刊样式
%	
%	6. acm，美国计算机学会期刊样式
%	
%	7. siam，美国工业和应用数学学会期刊样式	
	
%\begin{thebibliography}{100}
%{	\small 
%	\setlength{\baselineskip}{0.14in}
%	\parskip=0pt		

\printbibliography
%\bibliographystyle{abbrv}
%\bibliography{KraaijBib}
\end{document}